\documentclass[10pt]{amsart}
\usepackage{amsmath,amssymb,latexsym, amsfonts, amscd, amsthm}
\usepackage{color, tabu}

\usepackage{dynkin-diagrams}
\usepackage{tikz-cd}

\usepackage{tikz}
\theoremstyle{plain}

\newtheorem{theorem}{Theorem}[section]

\newtheorem{proposition}[theorem]{Proposition}

\newtheorem{lemma}[theorem]{Lemma}

\theoremstyle{definition}
\newtheorem{definition}[theorem]{Definition}
\newtheorem{remark}[theorem]{Remark}

\def\G{{\bf G}}
\def\E{{\bf E}}      
\def\F{{\bf F}}

\def\A{\mathbf{A}}
\def\B{\mathbf{B}}
\def\C{\mathbf{C}}
\def\D{\mathbf{D}}
\def\E{\mathbf{E}}
\def\F{\mathbf{F}}
\def\G{\mathbf{G}}

\def\F{\mathbf{F}}

\def\C{\mathbf{C}}
\def\G{\mathbf{G}}

\def\bdf{\begin{defn}}
\def\edf{\end{defn}}

\usepackage{color}

\begin{document}

\title{The sections of the Weyl group}

\author{Moshe Adrian}
\address{Moshe Adrian:  Department of Mathematics, Queens College, CUNY, Queens, NY 11367-1597}
\email{moshe.adrian@qc.cuny.edu}

\maketitle

\begin{abstract}
We compute all sections of the Weyl group, that satisfy the braid relations, in the case that $G$ is an almost-simple connected reductive group defined over an algebraically closed field.  We then demonstrate that this set of sections gives rise to an interesting partially ordered structure, and also give some applications.
\end{abstract}

\section{Introduction}\label{intro}
Let $G$ be a connected reductive algebraic group over an algebraically closed field $F$.  Let $T$ be a maximal torus in $G$, and let $W = N / T$ denote the associated Weyl group, where $N$ denotes the normalizer of $T$ in $G$.  Let $\Delta$ be a set of simple roots of $T$ in $G$, and if $\alpha \in \Delta$, let $s_{\alpha} \in W$ be the simple reflection associated to $\alpha$.

In \cite{Tit66}, Tits canonically defined a section of the map $\pi : N \rightarrow W$, which satisfies the braid relations.  This section, which we denote $\mathcal{N}_{\circ}$, has been used extensively (for recent examples, see \cite{Lus18, AH17, Ros16}), and we write $\sigma_\alpha = \mathcal{N}_{\circ}(s_\alpha)$.  In this work, we determine all sections of $\pi : N \rightarrow W$, that satisfy the braid relations, in the case that $G$ is almost-simple (we also sometimes refer to such a section as a ``a section of $W$'').  To be more precise, let $\mathcal{S} : W \rightarrow N$ be a section (for the rest of the paper, this means both that $\pi \circ \mathcal{S} = 1$ and that $\mathcal{S}$ satisfies the braid relations).  Equivalently, $\mathcal{S}$ satisfies the braid relations when restricted to a set of simple reflections.  If $\alpha \in \Delta$, then $\mathcal{S}(s_{\alpha}) = t_{\alpha} \sigma_\alpha$ for some $t_{\alpha} \in T$.  The braid relations then impose a set of conditions on the collection $\{ t_{\alpha} \}_{\alpha \in \Delta}$.  We compute these conditions explicitly for each type, thereby determining all sections of $\pi$.  

For example, let $G$ be of type $\G_2$, and fix the simple system given by $\alpha = e_1 - e_2, \beta  = -2e_1 + e_2 + e_3$.  Since $\G_2$ is simply connected, the cocharacter lattice is spanned by the set of coroots, so we may write $t_\alpha = \alpha^{\vee}(a) \beta^{\vee}(b), t_\beta = \alpha^{\vee}(c) \beta^{\vee}(d)$ for some $a, b, c, d \in F$.  The braid relation for the pair $t_\alpha \sigma_\alpha, t_\beta \sigma_\alpha$ simplifies to
$$t_{\alpha} s_{\alpha}(t_{\beta}) s_{\alpha} s_{\beta}(t_{\alpha}) s_{\alpha} s_{\beta} s_{\alpha}(t_\beta) s_{\alpha} s_{\beta} s_\alpha s_\beta(t_\alpha) s_\alpha s_\beta s_\alpha s_\beta s_\alpha(t_\beta) $$ $$= t_\beta s_\beta(t_\alpha) s_\beta s_\alpha(t_\beta) s_\beta s_\alpha s_\beta(t_\alpha) s_\beta s_\alpha s_\beta s_\alpha (t_\beta) s_\beta s_\alpha s_\beta s_\alpha s_\beta (t_\alpha).$$
One may compute that the above equality is equivalent to $b^2 = c^2 = 1$.  Therefore, the sections of $W$ for $\G_2$ are those whose evaluation on $s_\alpha, s_\beta$ are given by $\mathcal{S}(s_\alpha) = t_\alpha \sigma_\alpha, \mathcal{S}(s_\beta) = t_\beta \sigma_\beta$, where $t_\alpha = \alpha^{\vee}(a) \beta^{\vee}(b), t_\beta = \alpha^{\vee}(c) \beta^{\vee}(d)$, and $b^2 = c^2 = 1$.

We then interpret the set of sections in the following way.  Let $\mathcal{S}$ be a section.  If $\alpha \in \Delta$, we denote the order of $\mathcal{S}(s_{\alpha})$ by $\mathrm{o}(\mathcal{S}(s_{\alpha}))$.  We associate to $\mathcal{S}$ a labeled Dynkin diagram in the following way: given a simple root $\alpha$, we label its associated node in the Dynkin diagram by $\mathrm{o}(\mathcal{S}(s_\alpha))$.  We call such a labeling an \emph{order profile}.  We may then define a partial order on the set of order profiles as follows.  First, we can obtain a strict partial order on the set of all sections by saying that if $\mathcal{S}, \mathcal{S}'$ are two sections (for $G$), then we write that $\mathcal{S}' < \mathcal{S}$ if $\mathrm{o}(\mathcal{S}'(s_{\alpha})) > \mathrm{o}(\mathcal{S}(s_{\alpha}))$ for all $\alpha \in \Delta$.  In particular, $\mathcal{S}$ is ``more homomorphic'' than $\mathcal{S}'$.  This translates to a partial ordering on the set of order profiles.  We then show that for each $G$, there is a unique maximal order profile with respect to this partial ordering.  We call any section having this maximal order profile an \emph{optimal section}, and we say that such a section is a ``most homomorphic'' section.  

We also compute the set of $T$-conjugacy classes of sections for each type.  First note that $T$ acts naturally by conjugation on the set of sections; moreover, this action preserves order profiles.  To say that two sections $\mathcal{S}, \mathcal{S}'$ are $T$-conjugate is to say that there is a $t \in T$ such that $t \mathcal{S}(w) t^{-1} = \mathcal{S}'(w)$ for all $w \in W$.
A summary of some of the main results is the following.  For some types, the set of order profiles coincides with the set of $T$-conjugacy classes of sections; this happens in the cases of adjoint type $\C_n$, adjoint type $\D_n$, $\F_4$, $\G_2$, $\E_8$, adjoint type $\E_7$, and adjoint type $\E_6$.  Every other type has at least one order profile containing more than one $T$-conjugacy class of sections.

For an example of the above, we again consider $G$ of type $\G_2$, and let $\mathcal{S}$ be a section of $W$.  Then $\mathcal{S}(s_\alpha) = t_\alpha \sigma_\alpha, \mathcal{S}(s_\beta) = t_\beta \sigma_\beta$ have either order $2$ or $4$.  Moreover, one can compute that $t_\alpha \sigma_\alpha$ has order $2$ if and only if $b = -1$, and $t_\beta \sigma_\beta$ has order $2$ if and only if $c = -1$.  The Hasse diagram of the ensuing partially ordered set of order profiles then has four vertices and is in the shape of a diamond, with the unique maximal element being order $2$ on the lifts of $s_\alpha$ and $s_\beta$ (hence in particular a homomorphism $W \rightarrow N$).  To give an example regarding the labeling of the Dynkin diagram, the section which is order $2$ on $t_\alpha s_\alpha$ and order $4$ on $t_\beta s_\beta$ would be represented by the labeled Dynkin diagram that has the number $2$ on the $\alpha$ node and the number $4$ on the $\beta$ node.  For the $T$-conjugacy classes of these sections, we first note that if $\mathcal{S}(s_{\alpha}) = t_{\alpha} \sigma_{\alpha}$, and if $t \in T$, then $t t_{\alpha} \sigma_{\alpha} t^{-1} = t_{\alpha} t s_{\alpha}(t^{-1}) \sigma_\alpha$.  One can compute, then, that two sections $\mathcal{S}, \mathcal{S}'$ for $\G_2$ are conjugate by $T$ if and only if they have the same associated order profile.  

Finally, we present some applications of our results.  Let $H$ be a subgroup of $W$.  Let us say that $H$ \emph{lifts} if there is a homomorphism $\iota : H \hookrightarrow N$ that is compatible with $\pi : N \rightarrow W$; that is, such that $\pi \circ \iota = 1$.  Various papers have studied the question of whether $W$ lifts; see, for example \cite{AH17, CWW74}.  This question is equivalent to asking whether there is a section of $\pi$ such that the lifts of all simple reflections have order two.  Our computations (and in particular, a quick glance at \S\ref{summary}) show exactly which groups satisfy this condition.

On the other hand, it might be interesting to know whether other subgroups of $W$ lift. For example, in \cite{Adr18}, in order to determine whether the Kottwitz homomorphism for $p$-adic groups splits (i.e. has a homomorphic section), it was necessary to know whether a certain subgroup of $W$, which we call $\mathcal{J}$ (which is isomorphic to the fundamental group of $G$), lifts.  We attempted to show that $\mathcal{N}_{\circ}$ exhibits a lifting of $\mathcal{J}$, but this turned out to not be true in a small number of cases.  We were still nonetheless able to prove that $\mathcal{J}$ still lifts in those cases, in a roundabout way.  The current paper grew out of an attempt to show that there is a more natural lifting of $\mathcal{J}$ than the one we presented in \cite{Adr18}, thereby providing a more natural splitting of the Kottwitz homomorphism (see Theorem \ref{kottwitz}).  In particular, in \S\ref{application}, we prove that if $\mathcal{S}$ is an optimal section, then $\mathcal{S}$ exhibits a lifting of $\mathcal{J}$.  

The paper is organized as follows.  In \S\ref{prelim}, we recall the basic notions about sections of the Weyl group, as well as Tits' section.  In \S\ref{sections}, we explicitly compute and describe all sections of the Weyl group in types $\textbf{A}$ through $\textbf{G}$ (except in low rank, which we provide in a later section), as well as the set of $T$-conjugacy classes of sections.  In \S\ref{application}, we give our application to the Kottwitz homomorphism, and the lifting of $\mathcal{J}$.  Finally, in \S\ref{summary}, we present a table which shows all order profiles for each type.  In \S\ref{lowrank}, we present the results for low rank groups.

\subsection{Acknowledgements}
We would like to thank an anonymous referee of \cite{Adr18} for asking whether the Tits section could be modified in order to generalize the main result of \cite{Adr18}.  It was this question that eventually led to the current paper.  We also would like to thank Jeffrey Adams for numerous helpful discussions on the contents of this work.

\section{Preliminaries}\label{prelim}
We first remind the reader of the definition of Tits' section from \cite{Tit66}, as well as some generalities about general sections.  We follow \cite[\S8.1, \S9.3]{Spr98} closely.  Let $G$ be a connected reductive group over an algebraically closed field $F$, let $T$ be a maximal torus in $G$, and let $\Phi$ be the associated set of roots.  For each $\alpha \in \Phi$, let $s_{\alpha}$ denote the associated reflection in the Weyl group $W = N / T$.

\begin{proposition}\cite[Proposition 8.1.1]{Spr98}\label{properties}
\begin{enumerate}
\item For $\alpha \in \Phi$ there exists an isomorphism $u_{\alpha}$ of $\mathbf{G}_a$ onto a unique closed subgroup $U_{\alpha}$ of $G$ such that $t u_{\alpha}(x) t^{-1} = u_{\alpha}(\alpha(t) x) \ (t \in T, x \in F)$.
\item $T$ and the $U_{\alpha} \ (\alpha \in \Phi)$ generate $G$.
\end{enumerate}
\end{proposition}

Tits then defines a representative $\sigma_{\alpha}$, of $s_{\alpha}$, in $N$:

\begin{lemma}\cite[Lemma 8.1.4]{Spr98}\label{realization}
\begin{enumerate}
\item The $u_{\alpha}$ may be chosen such that for all $\alpha \in \Phi$, 
\[
\sigma_{\alpha} = u_{\alpha}(1) u_{-\alpha}(-1) u_{\alpha}(1)
\]
lies in $N$ and has image $s_{\alpha}$ in $W$.  For $x \in F^{\times}$, we have
\[
u_{\alpha}(x) u_{-\alpha}(-x^{-1}) u_{\alpha}(x) = \alpha^{\vee}(x) \sigma_{\alpha};
\]
\item $\sigma_{\alpha}^2 = \alpha^{\vee}(-1)$ and $\sigma_{-\alpha} = \sigma_{\alpha}^{-1}$;
\item If $u \in U_{\alpha} - \{1 \}$ there is a unique $u' \in U_{-\alpha} - \{1 \}$ such that $u u' u \in N$;
\item If $(u_{\alpha}')_{\alpha \in \Phi}$ is a second family with the property (1) of Proposition \ref{properties} and property (1) of Lemma \ref{realization}, there exist $c_{\alpha} \in F^{\times}$ such that
\[
u_{\alpha}'(x) = u_{\alpha}(c_{\alpha} x), \ c_{\alpha} c_{-\alpha} = 1 \ (\alpha \in \Phi, x \in F).
\]
\end{enumerate}
\end{lemma}

A family $(u_{\alpha})_{\alpha \in \Phi}$ with the properties (1) of Proposition \ref{properties} and Lemma \ref{realization} is called a \emph{realization} of the root system $\Phi$ in $G$ (see \cite[\S8.1]{Spr98}).

If $\Phi^+ \subset \Phi$ is a system of positive roots and $S$ is the associated set of simple reflections, we have:

\begin{proposition}\cite[Proposition 8.3.3]{Spr98}\label{braidrelations}
Let $\mu$ be a map of $S$ into a multiplicative monoid with the property: if $s,t \in S$, $s \neq t$, then 
$$\mu(s) \mu(t) \mu(s) \cdots = \mu(t) \mu(s) \mu(t) \cdots,$$
where in both sides the number of factors is $m(s,t)$.  Then there exists a unique extension of $\mu$ to $W$ such that if $s_1 \cdots s_h$ is a reduced decomposition for $w \in W$, we have $$\mu(w) = \mu(s_1) \cdots \mu(s_h).$$
\end{proposition}

We now fix a realization $(u_{\alpha})_{\alpha \in \Phi}$ of $\Phi$ in $G$.  Let $\alpha, \beta \in \Phi$ be linearly independent.  We denote $m(\alpha, \beta)$ the order of $s_{\alpha} s_{\beta}$.  Then $m(\alpha, \beta)$ equals one of the integers $2,3,4,6$.

\begin{proposition}\label{braid}\cite[Proposition 9.3.2]{Spr98}
Assume that $\alpha$ and $\beta$ are simple roots, relative to some system of positive roots.  Then
\[
\sigma_{\alpha} \sigma_{\beta} \sigma_{\alpha} \cdots = \sigma_{\beta} \sigma_{\alpha} \sigma_{\beta} \cdots,
\]
the number of factors on either side being $m(\alpha, \beta)$.
\end{proposition}

Following \cite[\S9.3.3]{Spr98}, we fix a set of positive roots $\Phi^+ \subset \Phi$, and let $\Delta$ be the associated set of simple roots.  Let $w = s_{\alpha_1} \cdots s_{\alpha_h}$ be a reduced expression for $w \in W$, with $\alpha_1, ..., \alpha_h \in \Delta$.  The element $\mathcal{N}_{\circ}(w) := \sigma_{\alpha_1} \cdots \sigma_{\alpha_h}$ is independent of the choice of reduced expression of $w$.  We therefore obtain a section $\mathcal{N}_{\circ} : W \rightarrow N$ of the homomorphism $N \rightarrow W$.  This is the section of Tits \cite{Tit66}. 

\section{The sections of the Weyl group for almost-simple groups}\label{sections}

In this section, we compute all sections of the Weyl group, classify the order profiles, and compute the set of $T$-conjugacy classes of sections.  Since the computations are quite lengthy, we only include some details in types $\A_n$ and exceptional groups. The details which we include illustrate all possible sorts of computations that arise in all types. 

Our method is as follows.  By Proposition \ref{braidrelations}, any section $\mathcal{S}$ of $W$ is determined by its values on a set of simple reflections.  Thus, let $\Delta$ be a set of simple roots.  If $\alpha \in \Delta$, let $s_{\alpha} \in W$ denote the associated simple reflection.  Then $\mathcal{S}(s_{\alpha}) = t_{\alpha} \sigma_{\alpha}$ for some $t_{\alpha} \in T$.  

Let $\alpha, \beta \in \Delta$.  In order that $t_{\alpha} \sigma_{\alpha}$ and $t_{\beta} \sigma_{\beta}$ satisfy the braid relations, it is necessary and sufficient that
\begin{equation}\label{prelimequation}
t_{\alpha} \sigma_{\alpha} t_{\beta} \sigma_{\beta} t_{\alpha} \sigma_{\alpha} \cdots = t_{\beta} \sigma_{\beta} t_{\alpha} \sigma_{\alpha} t_{\beta} \sigma_{\beta} \cdots,
\end{equation}
where in both sides the number of factors is $m(\alpha, \beta)$.  As $\sigma_{\alpha} t_{\beta} \sigma_{\alpha}^{-1} = s_{\alpha}(t_{\beta})$, and since the $\sigma$ satisfy the braid relations, \eqref{prelimequation} is equivalent to
\begin{equation}\label{mainequation}
t_{\alpha} s_{\alpha}(t_{\beta}) s_{\alpha} s_{\beta}(t_{\alpha}) \cdots = t_{\beta} s_{\beta}(t_{\alpha}) s_{\beta} s_{\alpha}(t_{\beta}) \cdots.
\end{equation}
Letting $\Delta = \{ \alpha_1, \alpha_2, ..., \alpha_n \}$, we set $t_i = t_{\alpha_i}$ and $s_i = s_{\alpha_i}$ for $i = 1, 2, ..., n$.  We choose a basis $\lambda_1, \lambda_2, ..., \lambda_n$ of $X_*(T)$, and then write the $t_i$ in terms of this basis.  We write $t_i = (a_{i,1}, a_{i,2}, ..., a_{i,n})$ to denote the element $\lambda_1(a_{i,1}) \lambda_2(a_{i,2}) \cdots \lambda_n(a_{i,n})$.  We then explicitly compute the equations \eqref{mainequation} for each pair of simple roots $\alpha, \beta$ in $\Delta$ in terms of our explicit coordinates for $t_{\alpha}, t_{\beta}$.  These equations impose strict conditions on the $a_{i,j}$, eventually yielding all possible sections.

Finally, we note that $T$ acts by conjugation on the set of sections.  We then compute the set of $T$-conjugacy classes of sections as follows.  If $t \in T$, we note that $t \mathcal{S}(s_{\alpha}) t^{-1} = t t_{\alpha} \sigma_{\alpha} t^{-1} = t_{\alpha} t s_{\alpha}(t^{-1}) \sigma_{\alpha}$.  If $\mathcal{S}, \mathcal{S}'$ are two sections, to say then $\mathcal{S}$ and $\mathcal{S}'$ are $T$-conjugate is to say that there is a $t \in T$ such that $t \mathcal{S}(w) t^{-1} = \mathcal{S}'(w) \ \forall w \in W$.  This is equivalent to saying that there is a $t \in T$ such that $t \mathcal{S}(s_{\alpha}) t^{-1} = \mathcal{S}'(s_{\alpha}) \ \forall \alpha \in \Delta$.  This imposes conditions on the sections $\mathcal{S}, \mathcal{S}'$, and we compute these conditions in all types.

In what follows, our choice for $\Delta$ will always be the convention in the ``Plates'' at the end of \cite{Bou02}.  Our choice of basis for $X_*(T)$ is as follows: if $G$ is simply connected and not adjoint, we always choose the basis $\lambda_1, ..., \lambda_n$ consisting of the standard set of simple coroots.  If $G$ is adjoint and not simply connected, we always choose the basis consisting of the standard set of fundamental coweights (which we denote $\omega_1, \omega_2, ..., \omega_n$) as in \cite{Bou02}.   If $G$ is both simply connected and adjoint, we always choose our basis to be the standard set of simple coroots.  If $G$ is neither simply connected nor adjoint, we define a basis case by case. 

In types $\textbf{A}_n$ through $\textbf{D}_n$, we will restrict ourselves first to the case $n \geq 6$, since some low rank cases yield different results.  We then carry out all computations in the low rank cases; the results for low rank are contained in \S\ref{lowrank}, but we do not include the computations (since they can easily be carried out by hand).

\subsection{Type $\A_{n}$}\label{A_n}

\subsubsection{Simply connected type}

We fix the standard set of roots $\alpha_1 = e_1 - e_2, ..., \alpha_{n-1} = e_{n-1} - e_n, \alpha_n = e_n - e_{n+1}$.  We recall that $m(\alpha_i, \alpha_{i+1}) = 3$ if $i = 1, ..., n-1$,  and otherwise the $m$ terms equal $2$.  Suppose that $G$ is simply connected.  Then the set of simple coroots form a basis for the cocharacter lattice.  As such, we may write an element of the corresponding maximal torus in terms of this basis, which we henceforth do.  For example, the tuple $(x_1, x_2, ..., x_n)$ will denote the element $$(e_1-e_2)^{\vee}(x_1) (e_2 - e_3)^{\vee}(x_2) \cdots (e_{n-2} - e_{n-1})^{\vee}(x_{n-2}) (e_{n-1} - e_n)^{\vee}(x_{n-1}) (e_n - e_{n+1})^{\vee}(x_n).$$

Suppose that $n \geq 6$.  A lengthy computation shows that the $m(\alpha, \beta) = 3$ equations yield that the $t_i$ are of the form

\begin{itemize}
\item $t_1 = (a_{1,1}, a_{1,2}, ..., a_{1,n})$
\item $t_2 = (a_{1,2}^{-1} a_{1,3}, a_{2,2}, a_{1,3}, a_{1,4}, ..., a_{1,n})$
\item $t_3 = (a_{1,2}^{-1} a_{1,3}, a_{1,2}^{-1} a_{1,4}, a_{3,3}, a_{1,4}, a_{1,5}, ..., a_{1,n})$
\item $t_4 = (a_{1,2}^{-1} a_{1,3}, a_{1,2}^{-1} a_{1,4}, a_{1,2}^{-1} a_{1,5}, a_{4,4}, a_{1,5}, a_{1,6}, ..., a_{1,n})$
\item ...
\item $t_{n-3} = (a_{1,2}^{-1} a_{1,3}, a_{1,2}^{-1} a_{1,4}, a_{1,2}^{-1} a_{1,5}, ..., a_{1,2}^{-1} a_{1,n-2}, a_{n-3, n-3}, a_{1,n-2}, a_{1,n-1}, a_{1,n})$
\item $t_{n-2} = (a_{1,2}^{-1} a_{1,3}, a_{1,2}^{-1} a_{1,4}, a_{1,2}^{-1} a_{1,5}, ..., a_{1,2}^{-1} a_{1,n-2}, a_{1,2}^{-1} a_{1,n-1} , a_{n-2, n-2}, a_{1,n-1}, a_{1,n})$
\item $t_{n-1} = (a_{1,2}^{-1} a_{1,3}, a_{1,2}^{-1} a_{1,4}, a_{1,2}^{-1} a_{1,5}, ..., a_{1,2}^{-1} a_{1,n-2}, a_{1,2}^{-1} a_{1,n-1}, a_{1,2}^{-1} a_{1,n}, a_{n-1, n-1}, a_{1,n})$
\item $t_n =  (a_{1,2}^{-1} a_{1,3}, a_{1,2}^{-1} a_{1,4}, a_{1,2}^{-1} a_{1,5}, ..., a_{1,2}^{-1} a_{1,n-2}, a_{1,2}^{-1} a_{1,n-1}, a_{1,2}^{-1} a_{1,n}, a_{1,2}^{-1}, a_{1,n})$
\end{itemize}

for some $a_{i,j} \in F$.  Turning to the $m(\alpha, \beta) =  2$ equations, it turns out that most of them are redundant.  Those that play a role are those coming from $m(\alpha_1, \alpha_j)$, with $j = 3, 4, ..., n$, yielding the conditions

\begin{itemize}
\item $a_{1,2} a_{1,4} = a_{1,3}^2$
\item $a_{1,3} a_{1,5} = a_{1,4}^2$
\item ...
\item $a_{1,{n-2}} a_{1,{n}} = a_{1,{n-1}}^2$
\item $a_{1,n-1} = a_{1,n}^2$
\end{itemize}
We remark again that we have omitted writing the many of the $m(\alpha, \beta) = 2$ equations above, since they end up being redundant.  Setting $a_{1,n} = a$, altogether we deduce
\begin{itemize}
\item $t_1 = (a_{1,1}, a^{n-1}, a^{n-2}, ..., a)$
\item $t_2 = (a^{-1}, a_{2,2}, a^{n-2}, a^{n-3}, ..., a)$
\item $t_3 = (a^{-1}, a^{-2}, a_{3,3}, a^{n-3}, a^{n-4}, ..., a)$
\item $t_4 = (a^{-1}, a^{-2}, a^{-3}, a_{4,4}, a^{n-4}, a^{n-5}, ..., a)$
\item ...
\item $t_{n-2} = (a^{-1}, a^{-2}, ..., a^{-(n-3)}, a_{n-2,n-2}, a^2, a)$
\item $t_{n-1} = (a^{-1}, a^{-2}, ..., a^{-(n-3)}, a^{-(n-2)}, a_{n-1,n-1}, a)$
\item $t_n = (a^{-1}, a^{-2}, ..., a^{-(n-3)}, a^{-(n-2)}, a^{-(n-1)}, a_{n,n})$.
\end{itemize}

To classify the order profiles, there are now two cases to consider.  Suppose that $n$ is odd.  It is straightforward to see that $(t_i \sigma_i)^2 \neq 1$ for all $i = 1, 2, ..., n$.  For example, 
\begin{align*}
& (t_2 \sigma_2)^2 = t_2 \sigma_2 t_2 \sigma_2^{-1} \sigma_2^2 = t_2 s_{\alpha_2}(t_2) \alpha_2^{\vee}(-1) \\
& = (a^{-1}, a_{2,2}, a^{n-2}, a^{n-3}, ..., a^2, a) (a^{-1}, a^{n-3} a_{2,2}^{-1}, a^{n-2}, a^{n-3}, ..., a) (1,-1,1,...,1) \\
& = (a^{-2}, -a^{n-3}, a^{2n-4}, a^{2n-6}, ..., a^4, a^2),
\end{align*}
which cannot equal $(1,1,...,1)$ since this would imply that $a^2 = 1$ and $a^{n-3} = -1$ (but recall that $n$ is odd).  
If, on the other hand, $n$ is even, then one can check that if $a = -1$, then $(t_i \sigma_i)^2 = 1 \ \forall i$.  This computation generalizes to general $(t_i \sigma_i)^2$.  

For classifying the $T$-conjugacy classes of sections, we note that if $t \in T$, then $t t_{\alpha} \sigma_{\alpha} t^{-1} = t_{\alpha} t s_{\alpha}(t^{-1}) n_{\alpha}$.  Then, if $t = (v_1, v_2, ..., v_n)$ (which we recall equals $(e_1-e_2)^{\vee}(v_1) (e_2 - e_3)^{\vee}(v_2) \cdots (e_{n-2} - e_{n-1})^{\vee}(v_{n-2}) (e_{n-1} - e_n)^{\vee}(v_{n-1}) (e_n - e_{n+1})^{\vee}(v_n)$), we have 

\begin{itemize}
\item $t_1 t s_1(t^{-1}) = (a_{1,1} v_1^2 v_2^{-1}, a^{n-1}, a^{n-2}..., a)$
\item $t_2 t s_2(t^{-1}) = (a^{-1}, a_{2,2} v_1^{-1} v_2^2 v_3^{-1}, a^{n-2}, a^{n-3}, ..., a)$
\item $t_3 t s_3(t^{-1})= (a^{-1}, a^{-2}, a_{3,3}v_2^{-1} v_3^2 v_4^{-1}, a^{n-3}, a^{n-4}, ..., a)$
\item $t_4 t s_4(t^{-1})= (a^{-1}, a^{-2}, a^{-3}, a_{4,4} v_3^{-1} v_4^2 v_5^{-1}, a^{n-4}, a^{n-5}, ..., a)$
\item ...
\item $t_n t s_n(t^{-1})=  (a^{-1}, a^{-2}, a^{-3}, ..., a^{-(n-3)}, a^{-(n-2)}, a^{-(n-1)}, a_{1,n} v_{n-1}^{-1} v_n^2)$
\end{itemize}

An analysis above the above formulas, together with the earlier results, yields

\begin{proposition} $ $
\begin{enumerate}
\item Let $n$ be odd. Then for every $m \in \mathbb{N}$, there is a section of the Weyl group such that the lift of every $s_i$ is order $4m$.  This section is given by setting $a$ to be a primitive $2m^{\mathrm{th}}$ root of unity.  In particular, the Tits section is optimal.  Moreover, $\mathcal{N}_{\circ}(s_i)^2 \neq 1 \ \forall i = 1, 2, ..., n$.

Let $n$ be even.  Then for every $m \in \mathbb{N}$, there is a section of the Weyl group such that the lift of every $s_i$ is order $2m$.  For $m$ odd, such a section is obtained by taking $a$ to be a primitive $2m^{\mathrm{th}}$ root of unity.  For $m$ even,  such a section is obtained by taking $a$ to be a primitive $m^{\mathrm{th}}$ root of unity. In particular, the section given by $\mathcal{S}(s_i) = t_i \sigma_i$ by setting $a = -1$, is an optimal section (in fact a homomorphic section).

\item Two sections are $T$-conjugate if and only if they have the same value of $a$.
\end{enumerate}
\end{proposition}

\begin{remark}
If $n$ is odd or even, then there is a section such that the lift of every $s_i$ has infinite order, given for example by choosing $a$ to not be a primitive root of unity.
\end{remark}

\subsubsection{Neither simply connected nor adjoint type}\label{nonsimpnonadjAn}

Recall that the isogenies of type $\A_n$ are in one to one correspondence with the subgroups of $\mathbb{Z} / (n+1) \mathbb{Z}$.  Let $a,b \in \mathbb{N}$ such that $n+1 = ab$, and let $$\omega_a = \frac{1}{n+1} [ (n+1-a)(\alpha_1 + 2 \alpha_2 + ... + (a-1) \alpha_{a-1}) + a((n+1 - a) \alpha_a + (n-a) \alpha_{a+1} + ... + \alpha_n)]$$ be the standard fundamental coweight corresponding to the integer $a$.  Then each isogeny that is neither simply connected nor adjoint has cocharacter lattice given by $X_*(T) = \langle Q^{\vee}, \omega_a \rangle$ for a positive integer $a$ such that $a$ divides $n+1$, $a \neq 1$, and $a \neq n+1$, and $Q^{\vee}$ is the coroot lattice.  Then, a basis for this cocharacter lattice is given by
$$\lambda_1 = \alpha_1^{\vee}, \lambda_2 = \alpha_2^{\vee}, ..., \lambda_{n-1} = \alpha_{n-1}^{\vee}, \lambda_n = \omega_a$$

Indeed (and we will need this calculation soon), $$\alpha_n^{\vee} = b \omega_a - (b-1) \alpha_1^{\vee} - 2(b-1) \alpha_2^{\vee} - ... - a(b-1) \alpha_a^{\vee} - (n-a) \alpha_{a+1}^{\vee} - (n-a-1) \alpha_{a+2}^{\vee} - ... - 3 \alpha_{n-2}^{\vee} - 2 \alpha_{n-1}^{\vee}.$$

We continue to write an element of the torus in terms of the above basis.  For example, the tuple $(x_1, x_2, ..., x_n)$ denotes the element $\lambda_1(x_1) \lambda_2 (x_2) \cdots \lambda_n(x_n)$.  We record the following actions: $s_i(\lambda_j) = \lambda_j \ \forall i,j$ except in the following cases:

\begin{itemize}
\item $s_i(\lambda_i) = -\lambda_i \ \forall i = 1,2, ..., n-1$
\item $s_i(\lambda_n) = \lambda_n \ \forall i \neq a$
\item $s_a(\lambda_n) = \lambda_n - \lambda_a$
\item $s_i(\lambda_{i+1}) = \lambda_i + \lambda_{i+1}$ for $i = 1, 2, ..., n-2$
\item $s_{i+1}(\lambda_i) = \lambda_i + \lambda_{i+1}$ for $i = 1, 2, ..., n-2$
\item $s_{n}(\lambda_{n-1}) = \alpha_{n-1} + \alpha_n = \lambda_{n-1} + b \lambda_n - (b-1) \lambda_1 - (2b-2) \lambda_2 - ... - a(b-1) \lambda_a - (n-a) \lambda_{a+1} - (n-a-1) \lambda_{a+2} - ... - 3 \lambda_{n-2} - 2 \lambda_{n-1}$.
\end{itemize}

After a long set of computations, the $m(\alpha, \beta)$ conditions give

\begin{itemize}
\item $t_1 = (a_{1,1}, x^{a-2}, x^{a-3}, ..., x, 1, 1, ..., 1, x)$
\item $t_2 = (x^{-1}, a_{2,2}, x^{a-3}, x^{a-4}, ..., x, 1, 1, ..., 1, x)$
\item $t_3 = (x^{-1}, x^{-2}, a_{3,3}, x^{a-4}, x^{a-5}, ..., x, 1, 1, ..., 1, x)$
\item ...
\item $t_{a-3} = (x^{-1}, x^{-2}, ..., x^{-(a-4)}, a_{a-3, a-3}, x^2, x, 1, 1, ..., 1, x)$
\item $t_{a-2} = (x^{-1}, x^{-2}, ..., x^{-(a-4)}, x^{-(a-3)}, a_{a-2, a-2}, x, 1, 1, ..., 1, x)$
\item $t_{a-1} = (x^{-1}, x^{-2}, ..., x^{-(a-4)}, x^{-(a-3)}, x^{-(a-2)}, a_{a-1, a-1}, 1, 1, ..., 1, x)$
\item $t_a = (x^{-1}, x^{-2}, ...., x^{-(a-4)}, x^{-(a-3)}, x^{-(a-2)}, x^{-(a-1)}, a_{a,a}, 1, 1, ..., 1, x)$
\item $t_{a+1} = (x^{-1}, x^{-2}, ...., x^{-(a-4)}, x^{-(a-3)}, x^{-(a-2)}, x^{-(a-1)}, x^{-a}, a_{a+1, a+1}, 1, 1, ..., 1, x)$
\item $t_{a+2} = (x^{-1}, x^{-2}, ...., x^{-(a-4)}, x^{-(a-3)}, x^{-(a-2)}, x^{-(a-1)}, x^{-a}, x^{-a}, a_{a+2, a+2}, 1, 1, ..., 1, x)$
\item ...
\item $t_{n-2} = (x^{-1}, x^{-2}, ...., x^{-(a-4)}, x^{-(a-3)}, x^{-(a-2)}, x^{-(a-1)}, x^{-a}, x^{-a}, x^{-a}, ..., x^{-a}, a_{n-2, n-2}, 1, x)$
\item $t_{n-1} = (x^{-1}, x^{-2}, ...., x^{-(a-4)}, x^{-(a-3)}, x^{-(a-2)}, x^{-(a-1)}, x^{-a}, x^{-a}, x^{-a}, ..., x^{-a}, a_{n-1, n-1}, x)$
\item $t_n = (x^{-1} y^{b-1}, x^{-2} y^{2(b-1)}, ..., x^{-(a-1)} y^{(a-1)(b-1)}, x^{-a} y^{a(b-1)}, x^{-a} y^{n-a}, x^{-a} y^{n-a-1}, ..., x^{-a} y^4, a_{n, n-2}, a_{n, n-1}, xy^{-b})$
\end{itemize}
where $y = a_{n, n-2} a_{n, n-1}^{-1}$ and $a_{n, n-1}^3 x^a = a_{n, n-2}^2$.  

We remark that the relations that come from the $m(\lambda_1, \lambda_2) = 3$ case are slightly different in the cases $a \neq 2$ and $a = 2$.  However, this difference results in the same above formulas for $t_i$.  Moreover, to be clear, we emphasize that if $a = 2$, then $t_1 = (a_{1,1}, 1, 1, ..., 1, x)$, $t_2 = (x^{-1}, a_{2,2}, 1, 1, ..., 1, x)$, etc.

We now seek to understand the orders of $t_i \sigma_i$.  Recalling that $t_i \sigma_i$ must have even order, we begin by computing that
\begin{align*}
& (t_1 \sigma_1)^2 = t_1 s_1(t_1) \sigma_1^2 = t_1 s_1(t_1) \alpha_1^{\vee}(-1)\\
& = (a_{1,1}, x^{a-2}, x^{a-3}, ..., x, 1, 1, ..., 1, x) \cdot (a_{1,1}^{-1} x^{a-2}, x^{a-2}, x^{a-3}, ..., x, 1, 1, ..., 1, x) (-1,1,1,...,1)\\
& = (-x^{a-2}, x^{2(a-2)}, x^{2(a-3)}, ..., x^2, 1, 1, ..., 1, x^2),
\end{align*}
which we note is trivial if and only if $x = -1$ and $a$ is odd.  Similar computations yield that if $i = 2, 3, ..., n-1$, then $(t_i \sigma_i)^2 = 1$ if and only if $a$ is odd and $x = -1$.

Lastly, a computation shows that
\begin{align*}
& (t_n \sigma_n)^2 = t_n s_n(t_n) \alpha_n^{\vee}(-1)\\
&= (x^{-2} x^{a(b-1)}, x^{-4} x^{2a(b-1)}, ..., x^{-2(a-1)} x^{(a-1)a(b-1)}, x^{-2a} x^{a^2(b-1)}, x^{-2a} x^{a(n-a)}, x^{-2a} x^{a(n-a-1)}, x^{-2a} x^{a(n-a-2)}, \\
& ...,x^{-2a} x^{4a}, x^a, 1, x^2 x^{-ab}) \alpha_n^{\vee}(-1).
\end{align*}
As discussed earlier, we have that  $\alpha_n^{\vee} = b \omega_a - (b-1) \alpha_1^{\vee} - 2(b-1) \alpha_2^{\vee} - ... - a(b-1) \alpha_a^{\vee} - (n-a) \alpha_{a+1}^{\vee} - (n-a-1) \alpha_{a+2}^{\vee} - ... - 3 \alpha_{n-2}^{\vee} - 2 \alpha_{n-1}^{\vee}$.  Therefore, 
\begin{align*}
& (t_n \sigma_n)^2 \\
& = ((-1)^{-(b-1)} x^{-2} x^{a(b-1)}, (-1)^{-2(b-1)} x^{-4} x^{2a(b-1)}, ..., (-1)^{-a(b-1)} x^{-2a} x^{a^2(b-1)}, (-1)^{-(n-a)} x^{-2a} x^{a(n-a)}, \\ 
& (-1)^{-(n-a-1)} x^{-2a} x^{a(n-a-1)}, (-1)^{-(n-a-2)}x^{-2a} x^{a(n-a-2)}, ...,(-1)^{-4} x^{-2a} x^{4a}, (-1)^{-3} x^a, (-1)^{-2}, (-1)^b x^2 x^{-ab}).
\end{align*}

Considering the third to last entry, we require $x^a = -1$ in order for $(t_n \sigma_n)^2$ to be trivial. We may conclude that $(t_n \sigma_n)^2 = 1$ if and only if $a$ is odd and $x = -1$.  

Generalizing the previous computations, we conclude

\begin{proposition} $ $
\begin{enumerate}
\item If $a$ is even and $m \in \mathbb{N}$, there is a section $\mathcal{S}$ such that $\mathcal{S}(s_i)$ has order $4m$ for all $i$.   This section is given by setting $x$ to be a primitive $2m^{\mathrm{th}}$ root of unity.  In particular, the Tits section is optimal, and $\mathcal{N}_{\circ}(s_i)^2 \neq 1 \ \forall i$. 

If $a$ is odd and $m \in \mathbb{N}$, there is a section $\mathcal{S}$ such that $\mathcal{S}(s_i)$ has order $2m$ for all $i$. For $m$ odd, it is obtained by taking $x$ to be a primitive $2m^{\mathrm{th}}$ root of unity.  For $m$ even, we take $a$ to be a primitive $m^{\mathrm{th}}$ root of unity.  In particular, any section given by $\mathcal{S}(s_{\alpha_i}) = t_i \sigma_i$ where $x = -1$, is an optimal section (in fact an in fact a homomorphic section).
\item Two sections are $T$-conjugate if and only if they have the same value of $x$.
\end{enumerate}
\end{proposition}

\begin{remark}
If $a$ is odd or even, then there is a section such that the lift of every $s_i$ has infinite order, given for example by choosing $x$ to not be a primitive root of unity.
\end{remark}

\subsubsection{Adjoint type}

We fix the standard set of fundamental coweights $\omega_1, \omega_2, ..., \omega_n$ from \cite[Plate I]{Bou02}.  Suppose that $G$ is adjoint.  Then the set of fundamental coweights form a basis for the cocharacter lattice.  As such, we may write an element of the torus in terms of this basis, which we henceforth do.  For example, the tuple $(x_1, x_2, ..., x_n)$ denotes the element $\omega_1(x_1) \omega_2 (x_2) \cdots \omega_n(x_n)$.  We record the following actions: $s_i(\omega_j) = \omega_j \ \forall i \neq j, s_1(\omega_1) = \omega_2 - \omega_1, s_i(\omega_i) = \omega_{i-1} + \omega_{i+1} - \omega_i \ \forall i = 2, 3, ..., n-1, s_n(\omega_n) = \omega_{n-1} - \omega_n$.

The $m(\alpha, \beta) = 2$ equations imply that

\begin{itemize}
\item $t_1 = (a_{1,1}, a_{1,2}, 1, 1, ..., 1)$
\item $t_2 = (a_{2,1}, a_{2,2}, a_{2,3}, 1, 1, ..., 1)$
\item $t_3 = (1, a_{3,2}, a_{3,3}, a_{3,4}, 1, 1, ..., 1)$
\item $t_4 = (1, 1, a_{4,3}, a_{4,4}, a_{4,5}, 1, 1, ..., 1)$
\item ...
\item $t_{n-2} = (1, 1, ..., 1, a_{n-2, n-3}, a_{n-2, n-2}, a_{n-2, n-1}, 1)$
\item $t_{n-1} = (1, 1, ..., 1, a_{n-1, n-2}, a_{n-1, n-1}, a_{n-1, n})$
\item $t_n =  (1, 1, ..., 1, a_{n-1,n}, a_{n,n})$
\end{itemize}
Turning to the $m(\alpha, \beta) =  3$ equations, it turns out that some of them are redundant.  Those that play a role are

\begin{itemize}
\item $a_{i,i-1}^2 a_{i-1,i}^2 a_{i,i} a_{i-1,i-1} = 1$ for $i = 2, 3, ..., n$.  
\item $a_{i,i+2} = a_{i+1, i} a_{i+1, i+1} a_{i+1, i+2}$ for $i = 1, 2, ...,n-2$.
\item $a_{i+2, i} = a_{i+1, i} a_{i+1, i+1} a_{i+1, i+2}$ for $i = 1, 2, ..., n-2$.  
\end{itemize}

Altogether, we deduce that
\begin{itemize}
\item $t_1 = (a_{1,1}, a_{1,2}, 1, 1, ..., 1)$
\item $t_2 = (a_{2,2}^{-1} a_{2,3}^{-1}, a_{2,2}, a_{2,3}, 1, 1, ..., 1)$
\item $t_3 = (1, a_{3,3}^{-1} a_{3,4}^{-1}, a_{3,3}, a_{3,4}, 1, 1, ..., 1)$
\item $t_4 = (1, 1, a_{4,4}^{-1} a_{4,5}^{-1}, a_{4,4}, a_{4,5}, 1, 1, ..., 1)$
\item ...
\item $t_{n-2} = (1, 1, ..., 1, a_{n-2, n-2}^{-1} a_{n-2, n-1}^{-1}, a_{n-2, n-2}, a_{n-2, n-1}, 1)$
\item $t_{n-1} = (1, 1, ..., 1, a_{n-1, n-1}^{-1} a_{n-1, n}^{-1}, a_{n-1,n-1}, a_{n-1, n})$
\item $t_n =  (1, 1, ..., 1, a_{n, n-1}, a_{n,n})$
\end{itemize}
together with the conditions that
\begin{itemize}
\item $a_{1,1} a_{1,2}^2 = a_{2,2} a_{2,3}^2 = a_{3,3} a_{3,4}^2 = ... = a_{n-2, n-2} a_{n-2, n-1}^2 = a_{n-1, n-1} a_{n-1, n}^2 = a_{n,n}^{-1} a_{n,n-1}^{-2}$
\end{itemize}

We now seek to understand the orders of $t_i \sigma_i$.  We compute for example that
$$(t_1 \sigma_1)^2 = t_1 s_1(t_1) \sigma_1^2 = t_1 s_1(t_1) \alpha_1^{\vee}(-1) = (a_{1,1}, a_{1,2}, 1, 1, ..., 1) (a_{1,1}^{-1}, a_{1,1} a_{1,2}, 1, 1, ..., 1) \alpha_1^{\vee}(-1).$$
To compute $\alpha_1^{\vee}(-1)$ in terms of the fundamental coweights, write $\alpha_1^{\vee}(-1) = \omega_1(z_1) \omega_2(z_2) \cdots \omega_n(z_n)$.   Applying $\alpha_i$ to both sides of this equation, as $i$ varies, we deduce that $z_i = 1 \ \forall i \neq 2$ and $z_2 = -1$.  In other words, 
\[
(t_1 \sigma_1)^2 = (1, -a_{1,1}, a_{1,2}^2, 1, 1, ..., 1).
\]

Similar computations imply that 

\begin{itemize}
\item $(t_2 \sigma_2)^2 = (-a_{2,2}^{-1} a_{2,3}^{-2}, 1, -a_{2,2} a_{2,3}^2, 1, 1, ..., 1)$
\item $(t_3 \sigma_3)^2 = (1, -a_{3,3}^{-1} a_{3,4}^{-2}, 1, - a_{3,3} a_{3,4}^2, 1, 1, ..., 1)$
\item ...
\item $(t_{n-1} \sigma_{n-1})^2 = (1, 1, ..., 1, -a_{n-1, n-1}^{-1} a_{n-1, n}^{-2}, 1, -a_{n-1, n-1} a_{n-1, n}^2)$
\item $(t_n \sigma_n)^2 = (1, 1, ..., 1, -a_{n, n} a_{n,n-1}^2, 1)$
\end{itemize}
Recalling the conditions 
\begin{align*}
& a_{1,1} a_{1,2}^2 = a_{2,2} a_{2,3}^2 = a_{3,3} a_{3,4}^2 = ... = a_{n-2, n-2} a_{n-2, n-1}^2 = a_{n-1, n-1} a_{n-1, n}^2 = a_{n,n}^{-1} a_{n,n-1}^{-2},
\end{align*}
we conclude that if the condition $a_{1,1} a_{1,2}^2 = a_{2,2} a_{2,3}^2 = a_{3,3} a_{3,4}^2 = ... = a_{n-2, n-2} a_{n-2, n-1}^2 = a_{n-1, n-1} a_{n-1, n}^2 = a_{n,n} a_{n,n-1}^2 = -1$ is satisfied, then we obtain an optimal section, in fact a homomorphic section.  We have therefore proven

\begin{proposition} $ $
\begin{enumerate}
\item For each $m \in \mathbb{N}$, there is a section of the Weyl group such that the lift of every $s_i$ has order $2m$.  This section is obtained by choosing $a_{i,j}$ in any way that satisfy $-a_{1,1} a_{1,2}^2 = -a_{2,2} a_{2,3}^2  = ... = -a_{n-1, n-1} a_{n-1, n}^2 = -a_{n,n} a_{n,n-1}^2$ to be a primtitive $m^{\mathrm{th}}$ root of unity.

In particular, any section given by $\mathcal{S}(s_{\alpha_i}) = t_i \sigma_i$ where $a_{1,1} a_{1,2}^2 = a_{2,2} a_{2,3}^2  = ... = a_{n-1, n-1} a_{n-1, n}^2 = a_{n,n} a_{n,n-1}^2 = -1$, is an optimal section (in fact a homomorphic section), whereas the Tits section is not.  We note that $\mathcal{N}_{\circ}(s_i)^2 \neq 1 \ \forall i$. 
\item Two sections are $T$-conjugate if and only if they have the same value of $a_{1,1} a_{1,2}^2$.
\end{enumerate}
\end{proposition}

\begin{remark}
There is a section such that the lift of every $s_i$ has infinite order, given for example by choosing $-a_{1,1} a_{1,2}^2$ to not be a primitive root of unity.
\end{remark}

\subsection{Type $\B_n$}\label{B_n}

\subsubsection{Simply connected type}

We fix the standard set of roots $\alpha_1 = e_1 - e_2, ..., \alpha_{n-1} = e_{n-1} - e_n, \alpha_n =  e_n$.  We may use the same methodology to compute the sections satisfying the braid relations as in the simply connected type $\A_n$ case, so we omit some of the details.  If $n \geq 6$, the result is that the $m(\cdot, \cdot)$ conditions yield 
\begin{itemize}
\item $t_1 = (a_{1,1}, 1, 1, 1, 1, ..., 1, a)$
\item $t_2 = (1, a_{2,2}, 1, 1,  .., 1,a)$
\item ...
\item $t_{n-1} = (1, 1, 1, 1, ..., 1, 1, 1, 1, a_{n-1,n-1}, a)$
\item $t_n = (a_{n,1}, 1, a_{n,1}, 1, ..., b_n, a_{n,n})$,
\end{itemize}
together with the condition that $a_{n,1}^2 = 1$ and $a^2 = 1$.  Here, $b_n$ equals $a_{n,1}$ if $n$ is even, and $b_n = 1$ if $n$ is odd.  

It is straightforward to see that $(t_i \sigma_i)^2 \neq 1$ if $1 \leq i < n$.  In the case that $n$ is even, we note that
$$(t_n \sigma_n)^2 = (1, 1, ..., 1, -a_{n,1}),$$  whereas if $n$ is odd, we note that 
$$(t_n \sigma_n)^2 = (1, 1, ..., 1, -1).$$
Therefore, we obtain

\begin{proposition}$ $
\begin{enumerate}
\item If $n$ is even, there are two order profiles of sections.  The first one satisfies that the lift of every $s_i$, for $i = 1, 2, ..., n-1$ is order $4$ and that the lift of $s_n$ is order $2$.  This section may be obtained by setting $a_{n,1} = -1$, and requiring that $a^2 = 1$.  The second order profile is given by the Tits section, in which the lifts of all $s_i$ are order $4$.

If $n$ is odd, there is only one order profile of sections, represented by the Tits section.  Moreover, the Tits section satisfies that all lifts of $s_i$ are order $4$.
\item Two sections are $T$-conjugate if and only if they have the same value of $a_{n,1}$ and $a$.
\end{enumerate}
\end{proposition}

\subsubsection{Adjoint type}
We fix the standard set of fundamental coweights $\omega_1, \omega_2, ..., \omega_n$ from \cite{Bou02}.  We record the following actions: $s_i(\omega_j) = \omega_j \ \forall i \neq j, s_1(\omega_1) = \omega_2 - \omega_1, s_i(\omega_i) = \omega_{i-1} + \omega_{i+1} - \omega_i \ \forall i = 2, 3, ..., n-2, s_{n-1}(\omega_{n-1}) = \omega_{n-2} + \omega_n - \omega_{n-1}, s_n(\omega_n) = 2 \omega_{n-1} - \omega_n$.

Suppose $n \geq 6$.  A lengthy computation shows that 
\begin{itemize}
\item $t_1 = (a_{1,1}, a_{1,2}, 1, 1, ..., 1, a_{1,n})$
\item $t_2 = (a_{2,2}^{-1} a_{2,3}^{-1}, a_{2,2}, a_{2,3}, 1, 1, ..., 1, a_{2,n})$
\item $t_3 = (1, a_{3,3}^{-1} a_{3,4}^{-1}, a_{3,3}, a_{3,4}, 1, 1, ..., 1, a_{3,n})$
\item ...
\item $t_{n-2} = (1, 1, ..., 1, a_{n-2, n-2}^{-1} a_{n-2, n-1}^{-1}, a_{n-2, n-2}, a_{n-2, n-1}, a_{n-1, n-2} a_{n-1, n-1} a_{n-1, n})$
\item $t_{n-1} = (1, 1, ..., 1, a_{n-1, n-2}, a_{n-1, n-1}, a_{n-1, n})$
\item $t_n =  (1, 1, ..., 1, a_{n, n-1}, a_{n,n})$
\end{itemize}
together with the conditions that
\begin{itemize}
\item $a_{1,n} = a_{2,n} = a_{3,n} = ... = a_{n-2,n}$
\item $a_{1,n}^2 = 1$
\item $a_{1,1} a_{1,2}^2 = a_{2,2} a_{2,3}^2 = a_{3,3} a_{3,4}^2 = ... = a_{n-2, n-2} a_{n-2, n-1}^2 = a_{n-1, n-1}^{-1} a_{n-1, n-2}^{-2}$
\item $a_{n-1, n}^4 a_{n-1, n-1}^2 = 1$
\item $a_{n, n-1}^2 a_{n,n}^2 = 1.$
\end{itemize}
We have, for example, that
\[
(t_{n-1} \sigma_{n-1})^2 = (1, 1, ..., 1, -a_{n-1, n-2}^2 a_{n-1, n-1}, 1, -a_{n-1, n-2}^{-2} a_{n-1, n-1}^{-1}).
\]

A computation, noting that $a_{n-1, n-1} a_{n-1, n-2}^2 = \pm 1$, then shows 

\begin{proposition} $ $
\begin{enumerate}
\item There are two order profiles of sections, given as follows.  Any section $\mathcal{S}$ given by $\mathcal{S}(s_{\alpha_i}) = t_i \sigma_i$ where $$a_{1,1} a_{1,2}^2 = a_{2,2} a_{2,3}^2 = a_{3,3} a_{3,4}^2 = ... = a_{n-2, n-2} a_{n-2, n-1}^2 = a_{n-1, n-1} a_{n-1, n-2}^2 = -1,$$ $a_{n-1, n}^4 a_{n-1, n-1}^2 = 1$, $a_{n, n-1}^2 a_{n,n}^2 = 1$, and $a_{1,n}^2 = a_{2,n}^2 = a_{3,n}^2 = ... = a_{n-2,n}^2 = 1$, is an optimal section, in fact a homomorphic section, whereas the Tits section is not.   We note that $\mathcal{N}_{\circ}(s_i)^2 \neq 1 \ \forall i = 1, 2, ..., n-1$, whereas $\mathcal{N}_{\circ}(s_n)^2 = 1$.  Moreover, one can compute that any section evaluated on $s_n$ has order $2$.
\item Two sections are $T$-conjugate if and only if they have the same value of $a_{1, n}$, and $a_{1,1} a_{1,2}^2$, and $a_{n,n-1} a_{n,n}$.
\end{enumerate}
\end{proposition}

\subsection{Type $\C_n$}\label{C_n}

\subsubsection{Simply Connected Type}

We fix the standard set of roots $\alpha_1 = e_1 - e_2, ..., \alpha_{n-1} = e_{n-1} - e_n, \alpha_n =  2e_n$.  Suppose $n \geq 6$.  A lengthy computation shows that if $n$ is even, we get 
\begin{itemize}
\item $t_1 = (a_{1,1}, b, a, b, a, ..., b, a, b)$
\item $t_2 = (ab, a_{2,2}, a, b, a, b, ..., a, b)$
\item $t_3 = (ab, 1, a_{3,3}, b, a, b, ..., a, b)$
\item $t_4 = (ab, 1, ab, a_{4,4}, a, b, a, ..., b, a, b)$
\item $t_5 = (ab, 1, ab, 1, a_{5,5}, b, a, b, ..., a, b)$
\item ...
\item $t_{n-2} = (ab, 1, ab, 1, ..., ab, 1, ab, a_{n-2,n-2}, a, b)$
\item $t_{n-1} = (ab, 1, ab, 1, ..., ab, 1, ab, 1, a_{n-1,n-1}, b)$
\item $t_n = (c, 1, c, 1, ..., c, a_{n,n})$,
\end{itemize}
together with the condition that $c^2 = 1$ and $a^2 = b^2 = 1$.

If $n$ is odd, we get 
\begin{itemize}
\item $t_1 = (a_{1,1}, a, b, a, ..., b, a, b)$
\item $t_2 = (ab, a_{2,2}, b, a, b, ..., a, b)$
\item $t_3 = (ab, 1, a_{3,3}, a, b, a, b, ..., a, b)$
\item $t_4 = (ab, 1, ab, a_{4,4}, b, a,b, a, ..., b, a, b)$
\item $t_5 = (ab, 1, ab, 1, a_{5,5}, a, b,a, b, ..., a, b)$
\item ...
\item $t_{n-2} = (ab, 1, ab, 1, ..., ab, 1, a_{n-2,n-2}, a, b)$
\item $t_{n-1} = (ab, 1, ab, 1, ..., ab, 1, ab, a_{n-1,n-1}, b)$
\item $t_n = (c, 1, c, 1, ..., c, 1, a_{n,n})$,
\end{itemize}
together with the condition that $c^2 = 1$ and $a^2 = b^2 = 1$.  

One may now compute
\begin{proposition} $ $
\begin{enumerate}
\item If $n$ is even, then the section given by $\mathcal{S}(s_{\alpha_i}) = t_i \sigma_i$ where $b = -1$ and $a^2 = c^2 = 1$ is an optimal section, whereas the Tits section is not.  If $n$ is odd, then the section given by $\mathcal{S}(s_{\alpha_i}) = t_i \sigma_i$ where $a = -1$ and $b^2 = c^2 = 1$ is an optimal section, whereas the Tits section is not.  In both cases, $\mathcal{S}(s_i)^2 = 1 \ \forall i = 1, 2, ..., n-1$, and $\mathcal{S}(s_n)^2 \neq 1$.  Moreover, one can compute that $\mathcal{N}_{\circ}(s_i)^2 \neq 1 \ \forall i$ in both cases ($n$ even and $n$ odd).  If $n$ is even or odd, there are only two order profiles of sections.
\item Two sections are $T$-conjugate if and only if they have the same value of $a$, $b$, and $c$.
\end{enumerate}
\end{proposition}

\subsubsection{Adjoint type}

We fix the standard set of fundamental coweights $\omega_1, \omega_2, ..., \omega_n$ from \cite{Bou02}.  We record the following actions: $s_i(\omega_j) = \omega_j \ \forall i \neq j, s_1(\omega_1) = \omega_2 - \omega_1, s_i(\omega_i) = \omega_{i-1} + \omega_{i+1} - \omega_i \ \forall i = 2, 3, ..., n-2, s_{n-1}(\omega_{n-1}) = \omega_{n-2} + 2 \omega_n - \omega_{n-1}, s_n(\omega_n) = \omega_{n-1} - \omega_n$.

Suppose $n \geq 6$.  A lengthy computation shows that 
\begin{itemize}
\item $t_1 = (a_{1,1}, a_{1,2}, 1, 1, ..., 1)$
\item $t_2 = (a_{2,2}^{-1} a_{2,3}^{-1}, a_{2,2}, a_{2,3}, 1, 1, ..., 1)$
\item $t_3 = (1, a_{3,3}^{-1} a_{3,4}^{-1}, a_{3,3}, a_{3,4}, 1, 1, ..., 1)$
\item $t_4 = (1, 1, a_{4,4}^{-1} a_{4,5}^{-1}, a_{4,4}, a_{4,5}, 1, 1, ..., 1)$
\item ...
\item $t_{n-2} = (1, 1, ..., 1, a_{n-2, n-2}^{-1} a_{n-2, n-1}^{-1}, a_{n-2, n-2}, a_{n-2, n-1}, 1)$
\item $t_{n-1} = (1, 1, ..., 1, a_{n-1, n-2}, a_{n-1, n-1}, a_{n-1, n-2}^{-2} a_{n-1, n-1}^{-2})$
\item $t_n =  (1, 1, ..., 1, a_{n, n-1}, a_{n, n-1}^{-2})$
\end{itemize}
together with the conditions that
\begin{itemize}
\item $a_{1,1} a_{1,2}^2 = a_{2,2} a_{2,3}^2 = a_{3,3} a_{3,4}^2 = ... = a_{n-2, n-2} a_{n-2, n-1}^2 = a_{n-1, n-1} a_{n-1, n-2}^2$
\item $(a_{n-1, n-1} a_{n-1, n-2}^2)^2 = 1$
\end{itemize}

After computations, one may conclude

\begin{proposition} $ $
\begin{enumerate}
\item Any section given by $\mathcal{S}(s_{\alpha_i}) = t_i \sigma_i$ where $$a_{1,1} a_{1,2}^2 = a_{2,2} a_{2,3}^2 = a_{3,3} a_{3,4}^2 = ... = a_{n-2, n-2} a_{n-2, n-1}^2 = a_{n-1, n-1} a_{n-1, n-2}^2 = -1,$$ is an optimal section, whereas the Tits section is not.  Moreover, $\mathcal{S}(s_i)^2 = 1 \ \forall i = 1, 2, ..., n-1$, and $\mathcal{S}(s_n)^2 \neq 1$.  We also note that $\mathcal{N}_{\circ}(s_i)^2 \neq 1 \ \forall i = 1,2, ..., n$.  There are only two order profiles of sections.
\item Two sections are $T$-conjugate if and only if they have the same value of $a_{1,1} a_{1,2}^2$.  In particular, two sections are $T$-conjugate if and only if they have the same order profile.
\end{enumerate}
\end{proposition}

\subsection{Type $\D_n$}\label{D_n}

\

Let $\omega_1, \omega_{n-1}, \omega_n$ be the fundamental coweights as in \cite[Plate IV]{Bou02}.  These coweights may be used to describe the non-simply connected, non-adjoint groups of type $\D_n$.  In particular, the cocharacter lattices of these isogenies are $\langle Q^{\vee}, \omega_1 \rangle, \langle Q^{\vee}, \omega_{n-1} \rangle,$ and $\langle Q^{\vee}, \omega_{n} \rangle$.  If $n$ is odd, the only such isogeny, up to isomorphism, is that given by $\omega_1$.  Thus, in considering the isogenies given by $\omega_{n-1}$ and $\omega_n$, we may assume that $n$ is even.

\subsubsection{Simply connected type}

We fix the standard set of roots $\alpha_1 = e_1 - e_2, ..., \alpha_{n-1} = e_{n-1} - e_n, \alpha_n = e_{n-1} + e_n$.  Suppose $n \geq 6$.  

Suppose that $n$ is even.  A lengthy computation shows that 
\begin{itemize}
\item $t_1 = (a_{1,1}, 1, ab, 1, ab, ..., 1, a,b)$
\item $t_2 = (ab, a_{2,2}, ab, 1, ab, 1, .., ab, 1, a,b)$
\item $t_3 = (ab, 1, a_{3,3}, 1, ab, 1, ..., ab, 1, a,b)$
\item $t_4 = (ab, 1, ab, a_{4,4}, ab, 1, ab, ..., 1, a,b)$
\item ...
\item $t_{n-3} = (ab, 1, ab, ...,1, a_{n-3,n-3}, 1, a,b) $
\item $t_{n-2} = (ab, 1, ab, 1, ..., ab, 1, ab, a_{n-2,n-2}, a,b)$
\item $t_{n-1} = (ab, 1, ab, 1, ..., ab, 1, ab, 1, a_{n-1,n-1}, b)$
\item $t_n = (ab, 1, ab, 1, ..., ab, 1, ab, 1, a, a_{n,n})$
\end{itemize}
such that $a^2 = b^2 = 1$.  

Suppose now that $n$ is odd.  A lengthy computation shows that 
\begin{itemize}
\item $t_1 = (a_{1,1}, 1, c^2, 1, c^2, ..., 1, c^2, c^{-1}, c)$
\item $t_2 = (c^2, a_{2,2}, c^2, 1, c^2, ..., 1, c^2, c^{-1}, c)$
\item $t_3 = (c^2, 1, a_{3,3}, 1, c^2, ..., 1, c^2, c^{-1}, c)$
\item $t_4 = (c^2, 1, c^2, a_{4,4}, c^2, 1, c^2, ..., 1, c^2, c^{-1}, c)$
\item ...
\item $t_{n-3} = (c^2, 1, c^2, ..., 1, c^2, a_{n-3,n-3}, c^2, c^{-1}, c) $
\item $t_{n-2} = (c^2, 1, c^2, ...,1, c^2, 1, a_{n-2,n-2}, c^{-1}, c)$
\item $t_{n-1} = (c^2, 1, c^2, ..., 1, c^2, 1, c^2, a_{n-1,n-1}, c)$
\item $t_n = (c^2, 1, c^2, ..., 1, c^2, c^{-1}, a_{n,n})$
\end{itemize}
such that $c^4 = 1$.

We may then conclude

\begin{proposition} $ $
\begin{enumerate}
\item If $n$ is even or odd, then the Tits section is optimal in the case of simply connected $\D_n$.  Moreover, one can compute that for any section $\mathcal{S}$, we have $\mathcal{S}(s_i)^2$ has order $4$ for all $i$.  Thus, there is only one order profile of sections, represented by the Tits section.
\item In the case that $n$ is even, two sections are $T$-conjugate if and only if they have the same $a$ and $b$.  In the case that $n$ is odd, two sections are $T$-conjugate if and only if they have the same $c$.
\end{enumerate}
\end{proposition}

\subsubsection{Middle isogeny with fundamental coweight $\omega_1$}
We consider the isogeny given by $X_*(T) = \langle Q^{\vee}, \omega_1 \rangle$ where $Q^{\vee}$ is the coroot lattice.  Then, a basis for this cocharacter lattice is given by
$$\lambda_1 = \alpha_1 = \alpha_1^{\vee}, \lambda_2 = \alpha_2^{\vee}, ..., \lambda_{n-1} = \alpha_{n-1}^{\vee}, \lambda_n = \omega_1$$
Indeed (and we will need this calculation soon), $$\alpha_n^{\vee} = \alpha_n = \lambda_{n-1} + 2 \lambda_n - 2(\lambda_1 + \lambda_2 + ... + \lambda_{n-1}).$$
We record the following actions: $s_i(\lambda_j) = \lambda_j \ \forall i,j$ except in the following cases:

$s_i(\lambda_i) = -\lambda_i \ \forall i = 1,2, ..., n-1$

$s_1(\lambda_n) = \lambda_n - \lambda_1$

$s_i(\lambda_{i+1}) = \lambda_i + \lambda_{i+1}$ for $i = 1, 2, ..., n-2$

$s_{i+1}(\lambda_i) = \lambda_i + \lambda_{i+1}$ for $i = 1, 2, ..., n-2$

$s_n(\lambda_{n-2}) = -2(\lambda_1 + \lambda_2 + .... + \lambda_{n-3}) - \lambda_{n-2} - \lambda_{n-1} + 2 \lambda_n$.

Let $n$ be even.  Then after a long set of computations, we obtain
\begin{itemize}
\item $t_1 = (a_{1,1}, 1, x, 1, x, ..., 1, x, 1, x, y)$
\item $t_2 = (xy^{-1}, a_{2,2}, x, 1, x, ..., 1, x, 1, x, y)$
\item $t_3 = (xy^{-1}, y^{-1}, a_{3,3}, 1, x, 1, ..., x, 1, x, 1, x, y)$
\item $t_4 = (xy^{-1}, y^{-1}, xy^{-1}, a_{4,4}, x, 1, x, .., 1, x, 1, x, y)$
\item ...
\item $t_{n-2} = (xy^{-1}, y^{-1}, xy^{-1}, ..., y^{-1}, xy^{-1}, a_{n-2, n-2}, x, y)$
\item $t_{n-1} = (xy^{-1}, y^{-1}, xy^{-1}, ..., y^{-1}, xy^{-1}, y^{-1}, a_{n-1, n-1}, y)$
\item $t_n = (xy^{-1} z^2, y^{-1} z^2, xy^{-1} z^2, ..., y^{-1} z^2, xy^{-1} z^2, y^{-1} z^2, xy z^2, y z^2, xz, y z^{-2})$
\end{itemize}
such that $x^2 = y^2 = 1$, with no condition on $z$

Let $n$ be odd.  Then after a long set of computations, we get

\begin{itemize}
\item $t_1 = (a_{1,1}, x, 1, x, ..., 1, x, 1, x, y)$
\item $t_2 = (xy^{-1}, a_{2,2}, 1, x, ..., 1, x, 1, x, y)$
\item $t_3 = (xy^{-1}, y^{-1}, a_{3,3}, x, 1, ..., x, 1, x, 1, x, y)$
\item $t_4 = (xy^{-1}, y^{-1}, xy^{-1}, a_{4,4}, 1, x, .., 1, x, 1, x, y)$
\item ...
\item $t_{n-2} = (xy^{-1}, y^{-1}, xy^{-1}, ..., y^{-1}, a_{n-2, n-2}, x, y)$
\item $t_{n-1} = (xy^{-1}, y^{-1}, xy^{-1}, ..., y^{-1}, xy^{-1}, a_{n-1, n-1}, y)$
\item $t_n = (xy^{-1} z^2, y^{-1} z^2, xy^{-1} z^2, ..., y^{-1} z^2, xy^{-1} z^2, y z^2, xy z^2, xz, y z^{-2})$
\end{itemize}
such that $x^2 = y^2 = 1$, with no condition on $z$.  One can conclude 

\begin{proposition} $ $
\begin{enumerate}
\item If $n$ is even, then any section given by $\mathcal{S}(s_{\alpha_i}) = t_i \sigma_i$ where $x^2 = 1$ and $y = -1$, is an optimal section (in fact a homomorphic section), whereas the Tits section is not.  In fact, $\mathcal{N}_{\circ}(s_i)^2 \neq 1 \ \forall i = 1, 2, ..., n.$  

If $n$ is odd, then any section given by $\mathcal{S}(s_{\alpha_i}) = t_i \sigma_i$ where $x^2 = y^2 = 1$ and $xy = -1$ is an optimal section (in fact a homomorphic section), whereas the Tits section is not.  In fact, $\mathcal{N}_{\circ}(s_i)^2 \neq 1 \ \forall i = 1, 2, ..., n.$  

If $n$ is even or odd, there are two order profile of sections.
\item Two sections are $T$-conjugate if and only if they have the same values of $x$ and $y$.
\end{enumerate}
\end{proposition}

\subsubsection{Middle isogeny with fundamental coweight $\omega_{n-1}$}

We consider the isogeny given by $X_*(T) = \langle Q^{\vee}, \omega_{n-1} \rangle$ where $Q^{\vee}$ is the coroot lattice.  As mentioned earlier, we may assume here that $n$ is even.  Then, a basis for this cocharacter lattice $X_*(T) = \langle Q^{\vee}, \omega_{n-1} \rangle$ is given by
$$\lambda_1 = \omega_{n-1}, \lambda_2 = \alpha_2^{\vee}, ..., \lambda_{n-1} = \alpha_{n-1}^{\vee}, \lambda_n = \alpha_n^{\vee}$$

Indeed (and we will need this calculation soon), $$\alpha_1^{\vee} = \alpha_1 = 2 \lambda_1 - 2 \lambda_2 - 3 \lambda_3 - 4 \lambda_4 - ... - (n-1) \lambda_{n-1} + \frac{n-2}{2}(\lambda_{n-1} - \lambda_n).$$
We record the following actions: $s_i(\lambda_j) = \lambda_j \ \forall i,j$ except in the following cases:

$s_i(\lambda_i) = -\lambda_i \ \forall i = 2, 3, ..., n$.

$s_i(\lambda_{i+1}) = \lambda_i + \lambda_{i+1}$ for $i = 2, 3, ..., n-2$

$s_{i+1}(\lambda_i) = \lambda_i + \lambda_{i+1}$ for $i = 2, 3, ..., n-2$

$s_1(\lambda_2) = 2 \lambda_1 - \lambda_2 - 3 \lambda_3 - 4 \lambda_4 - ... - (n-1) \lambda_{n-1} + \frac{n-2}{2}(\lambda_{n-1} - \lambda_n)$.

$s_{n-2}(\lambda_n) = \lambda_{n-2} + \lambda_n$

$s_{n-1}(\lambda_1) = \lambda_1 - \lambda_{n-1}$

$s_n(\lambda_{n-2}) = \lambda_{n-2} + \lambda_n$

Recall that $n$ is even.  Then after a long set of computations, we get
 
\begin{itemize}
\item $t_1 = (a^{-1}, a, b, a^2, ba, a^3, ba^2, a^4, ba^3, a^5, ..., ba^{\frac{n-2}{2} - 2}, a^{\frac{n-2}{2}}, c x^{\frac{n}{2}}, c x^{\frac{n-2}{2}})$
\item $t_2 = (1, a_{2,2}, 1, 1, ..., 1, c, c)$
\item $t_3 = (1, 1, a_{3,3}, 1, 1, ..., 1, c, c)$
\item $t_4 = (1, 1, 1, a_{4,4}, 1, 1, ..., 1, c, c)$
\item $t_5 = (1, 1, 1, 1, a_{5,5}, 1, 1, ..., 1, c, c)$
\item ...
\item $t_{n-3} = (1, 1, ..., 1, a_{n-3, n-3}, 1, c, c)$
\item $t_{n-2} = (1, 1, 1, ..., 1, a_{n-2, n-2}, c, c)$
\item $t_{n-1} = (1, 1, ..., 1, 1, a_{n-1, n-1}, c)$
\item $t_n = (1, 1, ..., 1, c, a_{n, n})$
\end{itemize}
such that $c^2 = 1, a^3 = b^2$, and where $x = a^{-1} b$. One concludes

\begin{proposition} $ $
\begin{enumerate}
\item The Tits section is optimal in the case of the isogeny of type $\D_n$ corresponding to the fundamental coweight $\omega_{n-1}$, with $n$ even.  Moreover, $\mathcal{N}_{\circ}(s_i)^2 \neq 1 \ \forall i = 1, 2, ..., n.$ Finally, one can compute that for any section $\mathcal{S}$, we have $\mathcal{S}(s_i)$ has order $4$ for all $i$, so that there is only one order profile of sections.
\item Two sections are $T$-conjugate if and only if they have the same value of $c$. 
\end{enumerate}
\end{proposition}

\subsubsection{Middle isogeny with fundamental coweight $\omega_n$}

We consider the isogeny given by $X_*(T) = \langle Q^{\vee}, \omega_n \rangle$ where $Q^{\vee}$ is the coroot lattice.  As mentioned earlier, we may assume here that $n$ is even. Then, a basis for this cocharacter lattice $X_*(T) = \langle Q^{\vee}, \omega_n \rangle$ is given by
$$\lambda_1 = \omega_n, \lambda_2 = \alpha_2^{\vee}, ..., \lambda_{n-1} = \alpha_{n-1}^{\vee}, \lambda_n = \alpha_n^{\vee}$$
Indeed (and we will need this calculation soon), $$\alpha_1^{\vee} = \alpha_1 = 2 \lambda_1 - 2 \lambda_2 - 3 \lambda_3 - 4 \lambda_4 - ... - (n-1) \lambda_{n-1} + \frac{n}{2}(\lambda_{n-1} - \lambda_n).$$

We record the following actions: $s_i(\lambda_j) = \lambda_j \ \forall i,j$ except in the following cases:

$s_i(\lambda_i) = -\lambda_i \ \forall i = 2, 3, ..., n$.

$s_i(\lambda_{i+1}) = \lambda_i + \lambda_{i+1}$ for $i = 2, 3, ..., n-2$

$s_{i+1}(\lambda_i) = \lambda_i + \lambda_{i+1}$ for $i = 2, 3, ..., n-2$

$s_1(\lambda_2) = \alpha_1 + \alpha_2 = 2 \lambda_1 - \lambda_2 - 3 \lambda_3 - 4 \lambda_4 - ... - (n-1) \lambda_{n-1} + \frac{n}{2} (\lambda_{n-1} - \lambda_n)$.

$s_{n-2}(\lambda_n) = \lambda_{n-2} + \lambda_n$

$s_{n}(\lambda_1) = \lambda_1 - \lambda_{n}$

$s_n(\lambda_{n-2}) = \lambda_{n-2} + \lambda_n$

 Recall that $n$ is even.  Then after a long set of computations, the braid relations give

\begin{itemize}
\item $t_1 = (a^{-1}, a, b, a^2, ba, a^3, ba^2, a^4, ba^3, a^5, ..., ba^{\frac{n-2}{2} - 2}, a^{\frac{n-2}{2}}, c x^{\frac{n-2}{2}}, c x^{\frac{n}{2}})$
\item $t_2 = (1, a_{2,2}, 1, 1, ..., 1, c, c)$
\item $t_3 = (1, 1, a_{3,3}, 1, 1, ..., 1, c, c)$
\item $t_4 = (1, 1, 1, a_{4,4}, 1, 1, ..., 1, c, c)$
\item ...
\item $t_{n-3} = (1, 1, ..., 1, a_{n-3, n-3}, 1, c, c)$
\item $t_{n-2} = (1, 1, 1, ..., 1, a_{n-2, n-2}, c, c)$
\item $t_{n-1} = (1, 1, ..., 1, 1, a_{n-1, n-1}, c)$
\item $t_n = (1, 1, ..., 1, c, a_{n, n})$
\end{itemize}
such that $c^2 = 1, a^3 = b^2$, and where $x = a^{-1} b$. One concludes that 

\begin{proposition} $ $
\begin{enumerate}
\item The Tits section is optimal.  In fact, $\mathcal{N}_{\circ}(s_i)^2 \neq 1 \ \forall i.$ Moreover, one can compute that for any section $\mathcal{S}$ satisfying the braid relations, we have $\mathcal{S}(s_i)$ has order $4$ for all $i$, so there is only one order profile of sections.
\item Two sections are $T$-conjugate if and only if they have the same value of $c$. 
\end{enumerate}
\end{proposition}

\subsubsection{Adjoint type}

We fix the standard set of fundamental coweights $\omega_1, \omega_2, ..., \omega_n$ from \cite{Bou02}.  We record the following actions: $s_i(\omega_j) = \omega_j \ \forall i \neq j, s_1(\omega_1) = \omega_2 - \omega_1, s_i(\omega_i) = \omega_{i-1} + \omega_{i+1} - \omega_i \ \forall i = 2, 3, ..., n-3, s_{n-2}(\omega_{n-2}) = \omega_{n-3} + \omega_{n-1} + \omega_n - \omega_{n-2}, s_{n-1}(\omega_{n-1}) = \omega_{n-2}  - \omega_{n-1}, s_n(\omega_n) = \omega_{n-2} - \omega_n$.

Altogether, we get 
\begin{itemize}
\item $t_1 = (a_{1,1}, a_{1,2}, 1, 1, ..., 1)$
\item $t_2 = (a_{2,2}^{-1} a_{2,3}^{-1}, a_{2,2}, a_{2,3}, 1, 1, ..., 1)$
\item $t_3 = (1, a_{3,3}^{-1} a_{3,4}^{-1}, a_{3,3}, a_{3,4}, 1, 1, ..., 1)$
\item $t_4 = (1, 1, a_{4,4}^{-1} a_{4,5}^{-1}, a_{4,4}, a_{4,5}, 1, 1, ..., 1)$
\item ...
\item $t_{n-3} = (1, 1, ..., 1, a_{n-3, n-3}^{-1} a_{n-3, n-2}^{-1}, a_{n-3, n-3}, a_{n-3, n-2}, 1, 1)$
\item $t_{n-2} = (1, 1, ..., 1, a_{n-2, n-2}^{-1} a_{n-2, n-1}^{-1}, a_{n-2, n-2}, a_{n-2, n-1}, a_{n-2, n-1})$
\item $t_{n-1} = (1, 1, ..., 1, a_{n-1, n-2}, a_{n-1, n-1}, a_{n-2, n-2} a_{n-2, n-1}^2)$
\item $t_n =  (1, 1, ..., 1, a_{n, n-2}, a_{n-2, n-2} a_{n-2, n-1}^2, a_{n,n})$
\end{itemize}
together with the conditions that
\begin{itemize}
\item $a_{1,1} a_{1,2}^2 = a_{2,2} a_{2,3}^2 = a_{3,3} a_{3,4}^2 = ... = a_{n-2, n-2} a_{n-2, n-1}^2 = a_{n-1, n-1}^{-1} a_{n-1, n-2}^{-2} = a_{n, n}^{-1} a_{n, n-2}^{-2}$
\item $(a_{n-2, n-2} a_{n-2, n-1}^2)^2 = 1$
\end{itemize}

We conclude that

\begin{proposition} $ $
\begin{enumerate}
\item Any section given by $\mathcal{S}(s_{\alpha_i}) = t_i \sigma_i$ where $$a_{1,1} a_{1,2}^2 = a_{2,2} a_{2,3}^2 = a_{3,3} a_{3,4}^2 = ... = a_{n-2, n-2} a_{n-2, n-1}^2 = a_{n-1, n-1} a_{n-1, n-2}^2 = a_{n,n} a_{n,n-2}^2 = -1,$$ is an optimal section (in fact a homomorphic section), whereas the Tits section is not.  Moreover, $\mathcal{N}_{\circ}(s_i)^2 \neq 1 \ \forall i = 1, 2, ..., n.$  One computes that there are two order profiles of sections; the one represented by the optimal section, and the one represented by the Tits section.
\item Two sections are $T$-conjugate if and only if they have the same value of $a_{1,1} a_{1,2}^2$. In particular, two sections are $T$-conjugate if and only if they are have the same order profile.
\end{enumerate}
\end{proposition}

\subsection{Type $\F_4$}\label{F_4}
We fix the standard set of roots $\alpha_1 = e_2 - e_3, \alpha_2 = e_3 - e_4, \alpha_3 = e_4, \alpha_4 = \frac{1}{2}(e_1 - e_2 - e_3 - e_4)$.  We recall that $m(\alpha_1, \alpha_2) = 3, m(\alpha_1, \alpha_3) = 2, m(\alpha_1, \alpha_4) = 2, m(\alpha_2, \alpha_3) = 4, m(\alpha_2, \alpha_4) = 2, m(\alpha_3, \alpha_4) = 3$.  Note that $G$ is simply connected.  Then the set of simple coroots form a basis for the cocharacter lattice.  As usual, we write an element of the maximal torus in terms of this basis.  For example, the tuple $(x_1, x_2, x_3, x_4)$ denotes the element $(e_2-e_3)^{\vee}(x_1) (e_3 - e_4)^{\vee}(x_2) e_4^{\vee}(x_3) (\frac{1}{2}(e_1 - e_2 - e_3 - e_4))^{\vee}(x_4)$.

A computation shows that the braid relations yields that 

\begin{itemize}
\item $t_1 = (a, 1, 1, 1)$
\item $t_2 = (1, b, 1, 1)$
\item $t_3 = (1, 1, c, x)$
\item $t_4 = (1, 1, x, d)$
\end{itemize}

such that $x^2 = 1$. It is straightforward to see that 

\begin{itemize}
\item $(t_1 \sigma_1)^2 = (-1, 1, 1, 1)$.
\item $(t_2 \sigma_2)^2 = (1, -1, 1, 1)$.
\item $(t_3 \sigma_3)^2 = (1, 1, -x, x^2)$.
\item $(t_4 \sigma_4)^2 = (1, 1, x^2, -x)$.
\end{itemize}

For example, 
\begin{align*}
& (t_3 \sigma_3)^2 = t_3 s_{\alpha_3}(t_3) \alpha_3^{\vee}(-1) \\
& = (1, 1, c, x) \cdot (1, 1, c^{-1} x, x) (1, 1, -1, 1) = (1, 1, -x, x^2).
\end{align*}
In particular, if we set $x = -1$, then we obtain that $(t_3 \sigma_3)^2 = (t_4 \sigma_4)^2 = 1$.  We have 

\begin{proposition} $ $
\begin{enumerate}
\item The optimal section in type $\F_4$ is given by  $\mathcal{S}(s_{\alpha_i}) = t_i \sigma_i$ where $x = -1$.  This section is not a homomorphism, and the Tits section is not optimal.  In fact, $\mathcal{S}(s_1)^2 \neq 1, \mathcal{S}(s_2)^2 \neq 1, \mathcal{S}(s_3)^2 = \mathcal{S}(s_4)^2 = 1$.  Moreover, $\mathcal{N}_{\circ}(s_i)^2 \neq 1 \ \forall i = 1,2, 3, 4.$  One sees that there are only two order profiles of sections.
\item Two sections are $T$-conjugate if and only if they have the same value of $x$.  In particular, two sections are $T$-conjugate if and only if they have the same order profile.
\end{enumerate}
\end{proposition}

\subsection{Type $\G_2$}\label{G_2}
We fix the standard set of roots $\alpha_1 = e_1 - e_2, \alpha_2 = -2e_1 +e_2 + e_3$.  We recall that $m(\alpha_1, \alpha_2) = 6$.  Note that $G$ is simply connected.  A computation shows that the $m(\alpha, \beta) = 6$ equation yields that 

\begin{itemize}
\item $t_1 = (a,b)$
\item $t_2 = (c,d)$
\end{itemize}
such that $c^2 = 1$ and $b^2 = 1$.  It is straightforward to see that $(t_1 \sigma_1)^2 = 1$ if and only if $b = -1$, and $(t_2 \sigma_2)^2 = 1$ if and only if $c = -1$.  Therefore,
\begin{proposition} $ $
\begin{enumerate}
\item The section given by $\mathcal{S}(s_i) = t_i \sigma_i$ with $c = -1, b = -1$ is optimal, in fact a homomorphic section.  Moreover, the Tits section is not optimal in the case of $\G_2$, and in fact $\mathcal{N}_{\circ}(s_1)^2 \neq 1, \mathcal{N}_{\circ}(s_2)^2 \neq 1$.  One can see that there are four order profiles of sections corresponding to the choice of the lifts of $s_1, s_2$ being order $2$ or order $4$.  We have that $(t_1 \sigma_1)^2 = 1$ if and only if $b = -1$, and $(t_2 \sigma_2)^2 = 1$ if and only if $c = -1$.  In particular, the section which lifts $s_2$ to an order $4$ element and $s_1$ to an order $2$ element is given by $b = -1$ and $c = 1$ and the section which lifts $s_1$ to an order $4$ element and $s_2$ to an order $2$ element is given by $b = 1$ and $c = -1$.
\item Two sections are $T$-conjugate if and only if they have the same value of $b$ and $c$.  In particular, two sections are $T$-conjugate if and only if they have the same order profile.
\end{enumerate}
\end{proposition}

\subsection{Type $\E_8$}\label{E_8}
In type $\E_8$, we will write $t_1 = (a_1, a_2, ..., a_8), t_2 = (b_1, b_2, ..., b_8), ..., t_8 = (h_1, h_2, ..., h_8)$, and similarly for types $\E_7, \E_6$.

We fix the standard set of roots $\alpha_1 = \frac{1}{2}(e_1 + e_8) - \frac{1}{2}(e_2 +e_3 + e_4 + e_5 + e_6 + e_7), \alpha_2 = e_1 + e_2, \alpha_3  = e_2 - e_1, \alpha_4 = e_3 - e_2, \alpha_5 = e_4 - e_3, \alpha_6 = e_5 - e_4, \alpha_7 = e_6 - e_5, \alpha_8 = e_7 - e_6$.  Note that $G$ is simply connected.  A lengthy computation shows that the $m(\alpha, \beta) = 2$ equations yield that 

\[
 \begin{tabu}{||c c c c c c c c||} 
 \hline
 t_1 & t_2 & t_3 & t_4 & t_5 & t_6 & t_7 & t_8 \\ [0.5ex] 
 \hline\hline
a_4 = a_2^2 & b_3 = b_1^2 & c_4 = c_2^2 & d_3 = d_1^2 & e_3 = e_1^2 & f_3 = f_1^2 & g_3 = g_1^2 & h_3 = h_1^2 \\ 
 \hline
 a_2 a_3 a_5 = a_4^2 & b_1 b_4 = b_3^2 & c_4 c_6 = c_5^2 & d_5 d_7 = d_6^2 & e_4 = e_2^2 & f_4 = f_2^2 & g_4 = g_2^2 & h_4 = h_2^2 \\
 \hline
a_4 a_6 = a_5^2 & b_4 b_6 = b_5^2 & c_5 c_7 = c_6^2 & d_6 d_8 = d_7^2 & e_1 e_4 = e_3^2 & f_1 f_4 = f_3^2 & g_1 g_4 = g_3^2 & h_1 h_4 = h_3^2\\
 \hline
a_5 a_7 = a_6^2 & b_5 b_7 = b_6^2 & c_6 c_8 = c_7^2 & d_7 = d_8^2 & e_6 e_8 = e_7^2 & f_2 f_3 f_5 = f_4^2 & g_2 g_3 g_5 = g_4^2 & h_2 h_3 h_5 = h_4^2\\
 \hline
a_6 a_8 = a_7^2 & b_6 b_8 = b_7^2 & c_7 = c_8^2 & & e_7 = e_8^2 & f_7 = f_8^2 & g_4 g_6 = g_5^2 & h_4 h_6 = h_5^2 \\ 
 \hline
a_7 = a_8^2 & b_7 = b_8^2 & & & & & & h_5 h_7 = h_6^2\\ 
 \hline
\end{tabu}
\]
The $m(\alpha, \beta) = 3$ equations yield
$$a_4 = c_1 a_3, \ b_3 b_5 = d_2 b_4, \ c_1 c_2 c_5 = d_3 c_4, \ d_2 d_3 d_6 = e_4 d_5, \ f_4 f_7 = f_5 e_6, \ g_5 g_8 = f_7 g_6, \ h_6 = h_7 g_8$$ and
\begin{itemize}
\item $c_i = a_i \ \forall i \neq 1,3$
\item $d_i = b_i \ \forall i \neq 2,4$
\item $d_i = c_i \ \forall i \neq 3,4$
\item $e_i = d_i \ \forall i \neq 4,5$
\item $f_i = e_i \ \forall i \neq 5,6$
\item $g_i = f_i \ \forall i \neq 6,7$
\item $h_i = g_i \ \forall i \neq 7,8$
\end{itemize}

The above conditions together imply that

\begin{itemize}
\item $t_1 = (a, 1, 1, 1, 1, 1, 1, 1)$
\item $t_2 = (1, b, 1, 1, 1, 1, 1, 1)$
\item $t_3 = (1, 1, c, 1, 1, 1, 1, 1)$
\item $t_4 = (1, 1, 1, d, 1, 1, 1, 1)$
\item $t_5 = (1, 1, 1, 1, e, 1, 1, 1)$
\item $t_6 = (1, 1, 1, 1, 1, f, 1, 1)$
\item $t_7 = (1, 1, 1, 1, 1, 1, g, 1)$
\item $t_8 = (1, 1, 1, 1, 1, 1, 1, h)$
\end{itemize}

It is straightforward to see that $(t_i \sigma_i)^2 \neq 1$ for all $i$.  In particular, 

\begin{proposition} $ $
\begin{enumerate}
\item The Tits section is optimal in the case of $\E_8$, and $\mathcal{N}_{\circ}(s_i)^2 \neq 1 \ \forall i = 1, 2, ..., 8$.  Moreover, there is only one order profile of sections.
\item All sections are $T$-conjugate.
\end{enumerate}
\end{proposition}

\subsection{Type $\E_7$}\label{E_7}
\subsubsection{Simply connected type}
We fix the standard set of roots $\alpha_1, \alpha_2, ..., \alpha_7$ of $\E_7$, taken from the case $\E_8$.  Suppose that $G$ is simply connected.  

The $m(\alpha, \beta) = 2$ equations yield  

\[
 \begin{tabu}{||c c c c c c c||} 
 \hline
 t_1 & t_2 & t_3 & t_4 & t_5 & t_6 & t_7  \\ [0.5ex] 
 \hline\hline
a_4 = a_2^2 & b_3 = b_1^2 & c_4 = c_2^2 & d_3 = d_1^2 & e_3 = e_1^2 & f_3 = f_1^2 & g_3 = g_1^2  \\ 
 \hline
 a_2 a_3 a_5 = a_4^2 & b_1 b_4 = b_3^2 & c_4 c_6 = c_5^2 & d_5 d_7 = d_6^2 & e_4 = e_2^2 & f_4 = f_2^2 & g_4 = g_2^2   \\
 \hline
a_4 a_6 = a_5^2 & b_4 b_6 = b_5^2 & c_5 c_7 = c_6^2 & d_6 = d_7^2 & e_1 e_4 = e_3^2 & f_1 f_4 = f_3^2 & g_1 g_4 = g_3^2  \\
 \hline
a_5 a_7 = a_6^2 & b_5 b_7 = b_6^2 & c_6  = c_7^2 & & e_6  = e_7^2 & f_2 f_3 f_5 = f_4^2 & g_2 g_3 g_5 = g_4^2  \\
 \hline
a_6  = a_7^2 & b_6  = b_7^2 &  & &  & & g_4 g_6 = g_5^2  \\ 
 \hline
\end{tabu}
\]
The $m(\alpha, \beta) = 3$ equations yield 
$$a_4 = c_1 a_3, \ b_3 b_5 = d_2 b_4, \ c_1 c_2 c_5 = d_3 c_4, \ d_2 d_3 d_6 = e_4 d_5, \ f_4 f_7 = f_5 e_6, \ g_5 = f_7 g_6,$$ and
\begin{itemize}
\item $c_i = a_i \ \forall i \neq 1,3$
\item $d_i = b_i \ \forall i \neq 2,4$
\item $d_i = c_i \ \forall i \neq 3,4$
\item $e_i = d_i \ \forall i \neq 4,5$
\item $f_i = e_i \ \forall i \neq 5,6$
\item $g_i = f_i \ \forall i \neq 6,7$
\end{itemize}
The above conditions together imply that

\begin{itemize}
\item $t_1 = (a_1, x, 1, 1, x , 1, x)$
\item $t_2 = (1, b_2, 1, 1, x, 1, x)$
\item $t_3 = (1, x, c_3, 1, x , 1, x)$
\item $t_4 = (1, x, 1, d_4, x , 1, x)$
\item $t_5 = (1, x, 1, 1, e_5, 1, x)$
\item $t_6 = (1, x, 1, 1, x, f_6, x)$
\item $t_7 = (1, x, 1, 1, x, 1, g_7)$
\end{itemize}

such that $x^2 = 1$.  It is straightforward to see that $(t_i \sigma_i)^2 \neq 1$ for every $i = 1, 2, ..., 7$.

\begin{proposition} $ $
\begin{enumerate}
\item The Tits section is optimal in the case of simply connected $\E_7$, and $\mathcal{N}_{\circ}(s_i)^2 \neq 1 \ \forall i$. Moreover, there is only one order profile of sections.
\item Two sections are $T$-conjugate if and only if they have the same value of $x$.
\end{enumerate}
\end{proposition}

\subsubsection{Adjoint type}

We fix the standard set of fundamental coweights $\omega_1, \omega_2, ..., \omega_7$ from \cite{Bou02}.  We record the following actions: $s_i(\omega_j) = \omega_j \ \forall i \neq j$, $s_1(\omega_1) = \omega_3 - \omega_1, s_2(\omega_2) = \omega_4 - \omega_2, s_3(\omega_3) = \omega_1 + \omega_4 - \omega_3, s_4(\omega_4) = \omega_2 + \omega_3 + \omega_5 - \omega_4, s_5(\omega_5) = \omega_4 + \omega_6 - \omega_5, s_6(\omega_6) = \omega_5 + \omega_7 - \omega_6, s_7(\omega_7) = \omega_6 - \omega_7$.

The $m(\alpha, \beta) = 2$ equations yield that all entries of the elements $t_1, t_2, ..., t_7$ are equal to $1$ except $a_1, a_3, b_2, b_4, c_1, c_3, c_4, d_2, d_3, d_4, d_5, e_4, e_5, e_6, f_5, f_6, f_7, g_6, g_7$.

The $m(\alpha, \beta) = 3$ equations yield
$$c_1^2 a_3^2 a_1 c_3 = 1, a_4 = c_1 c_3 c_4, d_2^2 b_4^2 b_2 d_4 = 1, b_3 = d_2 d_3 d_4, b_5 = d_2 d_4 d_5, d_1 = c_1 c_3 c_4, c_2 = d_2 d_3 d_4, d_3^2 c_4^2 d_4 c_3 = 1, $$ $$c_5 = d_3 d_4 d_5, e_2 = d_2 d_4 d_5, e_3 = d_3 d_4 d_5, e_4^2 d_5^2 e_5 d_4 = 1, d_6 = e_4 e_5 e_6, f_4 = e_4 e_5 e_6, f_5^2 e_6^2 f_6 e_5 = 1, e_7 = f_5 f_6 f_7, $$ $$g_5 = f_5 f_6 f_7, g_6^2 f_7^2 g_7 f_6 = 1,$$ and

\begin{itemize}
\item $c_i = a_i \ \forall i \neq 1,3,4$
\item $d_i = b_i \ \forall i \neq 2,3,4,5$
\item $c_6 = d_6, c_7 = d_7$
\item $e_1 = d_1, e_7 = d_7$
\item $f_1 = e_1, f_2 = e_2, f_3 = e_3$
\item $f_i = g_i \ \forall i \neq 5,6,7$
\end{itemize}

The above conditions together imply that

\begin{itemize}
\item $t_1 = (a^{-2}, 1, a, 1, 1, 1, 1)$
\item $t_2 = (1, b^{-2}, 1, b, 1, 1, 1)$
\item $t_3 = (c, 1, c^{-2}, c, 1, 1, 1)$
\item $t_4 = (1, d, d, d^{-2}, d, 1, 1)$
\item $t_5 = (1, 1, 1, e, e^{-2}, e, 1)$
\item $t_6 = (1, 1, 1, 1, f, f^{-2}, f)$
\item $t_7 = (1, 1, 1, 1, 1, g, g^{-2})$
\end{itemize}
with no conditions on $a, b, c, d, e, f, g$.  It is straightforward to see that $(t_i \sigma_i)^2 \neq 1$ for all $i$. 

\begin{proposition} $ $
\begin{enumerate}
\item The Tits section is optimal in the case of adjoint $\E_7$, and $\mathcal{N}_{\circ}(s_i)^2 \neq 1 \ \forall i.$  Moreover, one computes that there is only one order profile of sections.
\item All sections are $T$-conjugate.
\end{enumerate}
\end{proposition}

\subsection{Type $\E_6$}\label{E_6}
\subsubsection{Simply Connected Type}
We fix the standard set of roots $\alpha_1, \alpha_2, ..., \alpha_6$ of $\E_6$, taken from the case $\E_8$.  Suppose that $G$ is simply connected.  

The $m(\alpha, \beta) = 2$ equations yield that 

\[
 \begin{tabu}{||c c c c c c||} 
 \hline
 t_1 & t_2 & t_3 & t_4 & t_5 & t_6   \\ [0.5ex] 
 \hline\hline
a_4 = a_2^2 & b_3 = b_1^2 & c_4 = c_2^2 & d_3 = d_1^2 & e_3 = e_1^2 & f_3 = f_1^2   \\ 
 \hline
 a_2 a_3 a_5 = a_4^2 & b_1 b_4 = b_3^2 & c_4 c_6 = c_5^2 & d_5  = d_6^2 & e_4 = e_2^2 & f_4 = f_2^2    \\
 \hline
a_4 a_6 = a_5^2 & b_4 b_6 = b_5^2 & c_5  = c_6^2 &  & e_1 e_4 = e_3^2 & f_1 f_4 = f_3^2   \\
 \hline
a_5  = a_6^2 & b_5  = b_6^2 &  & &  & f_2 f_3 f_5 = f_4^2   \\
 \hline
\end{tabu}
\]
The $m(\alpha, \beta) = 3$ equations yield that
$$a_4 = c_1 a_3, \ b_3 b_5 = d_2 b_4, \ c_1 c_2 c_5 = d_3 c_4, \ d_2 d_3 d_6 = e_4 d_5, \ f_4 = f_5 e_6,$$ and

\begin{itemize}
\item $c_i = a_i \ \forall i \neq 1,3$
\item $d_i = b_i \ \forall i \neq 2,4$
\item $d_i = c_i \ \forall i \neq 3,4$
\item $e_i = d_i \ \forall i \neq 4,5$
\item $f_i = e_i \ \forall i \neq 5,6$
\end{itemize} 
The above conditions together imply that
\begin{itemize}
\item $t_1 = (a_1, 1, y, 1, y^2, y)$
\item $t_2 = (y^2, b_2, y, 1, y^2, y)$
\item $t_3 = (y^2, 1, c_3, 1, y^2, y)$
\item $t_4 = (y^2, 1, y, d_4, y^2, y)$
\item $t_5 = (y^2, 1, y, 1, e_5, y)$
\item $t_6 = (y^2, 1, y, 1, y^2, f_6)$
\end{itemize}

such that $y^3 = 1$. It is straightforward to see that $(t_i \sigma_i)^2 \neq 1$ for every $i = 1, 2, ..., 6$.  We compute, for example, that 
$$(t_1 \sigma_1)^2 = (-y, 1, y^2, 1, y, y^2).$$ We can conclude

\begin{proposition} $ $
\begin{enumerate}
\item The Tits section is optimal in the case of simply connected $\E_6$, and $\mathcal{N}_{\circ}(s_i)^2 \neq 1 \ \forall i.$  There is one other order profile of sections, given by $y$ is a non-trivial cube root of unity, which satisfies that the lifts of all $s_i$ have order $12$.
\item Two sections are $T$-conjugate if and only if they have the same value of $y$.
\end{enumerate}
\end{proposition}

\subsubsection{Adjoint Type}

We fix the standard set of fundamental coweights $\omega_1, \omega_2, ..., \omega_6$ from \cite{Bou02}.  We record the following actions: $s_i(\omega_j) = \omega_j \ \forall i \neq j$, $s_1(\omega_1) = \omega_3 - \omega_1, s_2(\omega_2) = \omega_4 - \omega_2, s_3(\omega_3) = \omega_1 + \omega_4 - \omega_3, s_4(\omega_4) = \omega_2 + \omega_3 + \omega_5 - \omega_4, s_5(\omega_5) = \omega_4 + \omega_6 - \omega_5, s_6(\omega_6) = \omega_5  - \omega_6$.

The $m(\alpha, \beta) = 2$ equations yield that all entries of the elements $t_1, t_2, ..., t_6$ are equal to $1$ except $a_1, a_3, b_2, b_4, c_1, c_3, c_4, d_2, d_3, d_4, d_5, e_4, e_5, e_6, f_5, f_6$.

The $m(\alpha, \beta) = 3$ equations yield
$$c_1^2 a_3^2 a_1 c_3 = 1, a_4 = c_1 c_3 c_4, d_2^2 b_4^2 b_2 d_4 = 1, b_3 = d_2 d_3 d_4, b_5 = d_2 d_4 d_5, d_1 = c_1 c_3 c_4, c_2 = d_2 d_3 d_4, d_3^2 c_4^2 d_4 c_3 = 1, $$ $$c_5 = d_3 d_4 d_5, e_2 = d_2 d_4 d_5, e_3 = d_3 d_4 d_5, e_4^2 d_5^2 e_5 d_4 = 1, d_6 = e_4 e_5 e_6, f_4 = e_4 e_5 e_6, f_5^2 e_6^2 f_6 e_5 = 1,$$ and

\begin{itemize}
\item $c_i = a_i \ \forall i \neq 1,3,4$
\item $d_i = b_i \ \forall i \neq 2,3,4,5$
\item $c_6 = d_6$
\item $e_1 = d_1$
\item $f_1 = e_1, f_2 = e_2, f_3 = e_3$
\end{itemize}

One can check that all of these conditions amount to 

\begin{itemize}
\item $t_1 = (a^{-2}, 1, a, 1, 1, 1)$
\item $t_2 = (1, b^{-2}, 1, b, 1, 1)$
\item $t_3 = (c, 1, c^{-2}, c, 1, 1)$
\item $t_4 = (1, d, d, d^{-2}, d, 1)$
\item $t_5 = (1, 1, 1, e, e^{-2}, e)$
\item $t_6 = (1, 1, 1, 1, f, f^{-2})$
\end{itemize}

It is straightforward to see that $(t_i \sigma_i)^2 \neq 1$ for every $i = 1, 2, ..., 6$.

\begin{proposition} $ $
\begin{enumerate}
\item The Tits section is optimal in the case of adjoint $\E_6$, and $\mathcal{N}_{\circ}(s_i)^2 \neq 1 \ \forall i = 1, 2, ..., 6.$  There is only one order profile of sections.
\item All sections are $T$-conjugate.
\end{enumerate}
\end{proposition}

\section{Application: splitting the Kottwitz homomorphism}\label{application}
Set $W^{\mathrm{ext}} = N_G(T) / T_{\circ}$, where $T_{\circ}$ is the maximal bounded subgroup of a maximal torus $T$ in $G$.  The group $W^{\mathrm{ext}}$ is the extended affine Weyl group, and we note that we have a semidirect product decomposition $W^{\mathrm{ext}} = X_*(T) \rtimes W$, where $X_*(T)$ is the cocharacter lattice of $T$.  
We also set $\Omega = W^{\mathrm{ext}} / W^{\mathrm{aff}}$, where $W^{\mathrm{aff}} = Q^{\vee} \rtimes W$ is the affine Weyl group and $Q^{\vee}$ is the coroot lattice.  We therefore have a canonical projection $N_G(T) \twoheadrightarrow \Omega$.  This projection is the Kottwitz homomorphism (see \cite{Kot97}).

It is known (see \cite[Proposition 1.18]{IM65} or \cite[Proposition 3.1]{Adr18}) that the elements of $\Omega$ may be represented by a collection of elements $\{1, \epsilon_i \rtimes w_i \} \subset  X_*(T) \rtimes W$, where $\epsilon_i$ are certain fundamental coweights.  In \cite{Adr18}, we considered the map
$$
\iota : \Omega \rightarrow N_G(T)
$$
$$
 (\epsilon_i, w_i) \mapsto \epsilon_i(\varpi^{-1}) \mathcal{N}_{\circ}(w_i),
$$
where $\varpi$ is a uniformizer in $F$.
The map $\iota$ is a section of the projection $N_G(T) \twoheadrightarrow \Omega$, and it turned out that $\iota$ is a homomorphism in all cases except the adjoint group of type $D_n$ where $n$ is odd, and some type $A_n$ cases.  Nonetheless, we were still able to use $\iota$ to construct a homomorphic section of the Kottwitz homomorphism, for all almost-simple $p$-adic groups (see \cite[Theorem 3.5]{Adr18}).  

The question about finding a homomorphic section boiled down to finding a homomorphic lift of the subgroup $\mathcal{J}$ of $W$ generated by the $w_i$.  We attempted to show that Tits' section $\mathcal{N}_{\circ}$ achieved this goal, but $\mathcal{N}_{\circ}$ turns out to not lift $\mathcal{J}$ in adjoint type $\D_n$ with $n$ odd, and some cases of type $\A_n$.  Here, we will show that if $\mathcal{S}$ is an optimal section, then $\mathcal{S}$ lifts $\mathcal{J}$.

We now recall a result about the map $\mathcal{N}_{\circ}$ from \cite{LS87}.

\begin{definition}
For $u,v \in W$, we define 
\[
\mathcal{F}(u,v) = \{\alpha \in \Pi \ | \ v(\alpha) \in -\Pi, u(v(\alpha)) \in \Pi \}.
\]
\end{definition}
The following proposition describes the failure of $\mathcal{N}_{\circ}$ to be a homomorphism (this proposition can also be found in \cite{Ros16}).
\begin{proposition}\cite[Lemma 2.1.A.]{LS87}\label{ahomomorphicity}
For $u,v \in W$,
\[
\mathcal{N}_{\circ}(u) \cdot \mathcal{N}_{\circ}(v) = \mathcal{N}_{\circ}(u \cdot v) \cdot \displaystyle\prod_{\alpha \in \mathcal{F}(u,v)} \alpha^{\vee}(-1).
\]
\end{proposition}

We now give a formula for computing powers of Tits' lifts.

\begin{definition}
For $w \in W, i \in \mathbb{N}$, we define
\[
\mathcal{F}_w(i) = \{ \alpha \in \Pi \ | \ w^i(\alpha) \in -\Pi, w^{i+1}(\alpha) \in \Pi  \}.
\]
\end{definition}

\begin{proposition}\label{powersformula}
If $w \in W$ and $n \in \mathbb{N}$, then
$$\mathcal{N}_{\circ}(w)^n = \mathcal{N}_{\circ}(w^n) \cdot \displaystyle\prod_{m = 1}^{n-1} \displaystyle\prod_{\alpha \in \mathcal{F}_w(m)} \alpha^{\vee}(-1).$$
\end{proposition}

\begin{proof}
By Proposition \ref{ahomomorphicity}, $\mathcal{N}_{\circ}(w)^2 = \mathcal{N}_{\circ}(w^2) \cdot \displaystyle\prod_{\alpha \in \mathcal{F}_w(1)} \alpha^{\vee}(-1).$  Multiplying by $\mathcal{N}_{\circ}(w)$ on the left and using Proposition \ref{ahomomorphicity} again, we get $\mathcal{N}_{\circ}(w)^3 = \mathcal{N}_{\circ}(w^3) \cdot \displaystyle\prod_{\alpha \in \mathcal{F}_w(2)} \alpha^{\vee}(-1)  \cdot \displaystyle\prod_{\alpha \in \mathcal{F}_w(1)} \alpha^{\vee}(-1).$ Continuing in this way, the claim follows.
\end{proof}

We proceed to show that an optimal section of $W$ lifts $\mathcal{J}$.  
First, in adjoint types $\E_6, \E_7$, and the isogenies of type $\D_n$ given by the fundamental coweights $\omega_n$, $\omega_{n-1}$, ($n$ even), we have shown in \S\ref{D_n}, \S\ref{E_6}, \S\ref{E_7} that the Tits section is optimal.  Moreover, we have shown in \cite{Adr18} that the Tits section lifts $\mathcal{J}$.  So we do not need to consider these cases.

It remains to check that the optimal section $\mathcal{S}$ lifts $\mathcal{J}$ in the cases of non-simply connected type $\A_n$, adjoint types $\B_n, \C_n, \D_n$, and the isogeny of 
$\D_n$ corresponding to the fundamental coweight $\omega_1$. 

We note that in the cases of non-simply connected non-adjoint type $\A_n$ with $n+1 = ab$ and $a$ is odd, as well as the cases of adjoint type $\A_n$, adjoint type $\B_n$, the isogeny of type $\D_n$ corresponding to the fundamental coweight $\omega_1$, and adjoint type $\D_n$, we have shown that the optimal section is a homomorphism, so that these cases also need not be considered.

In type $\A_n$ with $n+1 = ab$ and $a$ even, we have also shown in \S\ref{A_n} that the Tits section is optimal.  However, in \cite{Adr18}, we did not show that $\mathcal{N}_{\circ}$ lifts $\mathcal{J}$ (as we were able to prove the result that we needed in loc. cit. in a slightly different way), so we must consider this case below.

So it remains to check that an optimal section, which we denote $\mathcal{S}$, lifts $\mathcal{J}$ in the cases of non-simply connected non-adjoint type $\A_n$ with $n+1 = ab$ and $a$ even, and adjoint type $\C_n$.  We note that the Tits section in adjoint type $\C_n$ is not optimal, as shown in \S\ref{C_n}.

\subsection{Remaining $\A_n$ types} We now consider the group of type $\A_n$, non-simply connected, non-adjoint, $n + 1 = ab$ with $a$ even. We must check that if $w$ generates $\mathcal{J}$, that $\mathcal{S}(w)^r = \mathcal{S}(w^r)$ for all $r$.  But as we showed in \S\ref{nonsimpnonadjAn}, if $a$ is even, then the Tits section is optimal, so we may set $\mathcal{S} = \mathcal{N}_{\circ}$.

\begin{proposition}
If $w$ generates $\mathcal{J}$, then $\mathcal{N}_{\circ}(w)^r = \mathcal{N}_{\circ}(w^r) \ \forall r \in \mathbb{N}$.
\end{proposition}

\begin{proof}
The generator $w$ of $\mathcal{J}$ is the $a^{\mathrm{th}}$ power of the $n$-cycle $(1 \ 2 \ 3 \ \cdots \ n)$.  Denote this $a^{\mathrm{th}}$ power by $w_a$.  We need to show that $\mathcal{N}_{\circ}(w_a)^r = \mathcal{N}_{\circ}(w_a^r) \ \forall r \in \mathbb{N}$.  
By Proposition \ref{powersformula}, it is equivalent to show that 
\[
\displaystyle\prod_{i = 1}^{r-1} \displaystyle\prod_{\alpha \in \mathcal{F}_{w_a}(i)} \alpha^{\vee}(-1) = 1
\]
where $\mathcal{F}_{w_a}(i) = \{ \alpha \in \Pi \ | \ w_a^i(\alpha) \in -\Pi, w_a^{i+1}(\alpha) \in \Pi  \}$.  But a computation shows that
\[
\displaystyle\sum_{\alpha \in \mathcal{F}_{w_a}(i)} \alpha^{\vee} = ai (e_{n-a(i+1)+1} + e_{n-a(i+1)+2} + \cdots e_{n-ai}) - a(e_{n-ai+1} + e_{n-ai+2} + \cdots + e_n).
\]
Notice that the coefficients of every $e$ term is even, since $a$ is even.  Summing now from $i = 1$ to $i = r-1$, we obtain a summation of $e$ terms, each of whose coefficient is even.  In particular, this sum is twice a cocharacter, so it's evaluation on $-1$ is trivial, giving us the result that we seek.

\end{proof}

\subsection{Adjoint type $\C_n$} Consider the case of adjoint type $\C_n$, let $\mathcal{S}$ be any optimal section as in \S\ref{C_n}, and recall that $\omega_1, ..., \omega_n$ denote the fundamental coweights of type $\C_n$ as in \cite{Bou02}.  In particular $\mathcal{S}(s_{\alpha})^2 = 1$ for all simple roots $\alpha$ unless $\alpha = 2e_n$, in which case a computation shows that $\mathcal{S}(s_{2e_n})^2 = \alpha_n^{\vee}(-1) =  \omega_{n-1}(-1)$.  We have that $\mathcal{J}$ is of order $2$ and is generated by the Weyl element $w_{\Pi_n} w_{\Pi}$ (see  \cite[Proposition 1.18]{IM65}).  Here, $w_{\Pi_n}$ is the long element (from the Weyl group of type $\A_{n-1}$) obtained from removing the node $\alpha_n$ from the Dynkin diagram of $\C_n$, and $w_{\Pi}$ is the long element from type $\C_n$.  In other words,
\[
w_{\Pi_n} = s_1 s_2 s_1 s_3 s_2 s_1 s_4 s_3 s_2 s_1 \cdots s_{n-1} s_{n-2} \cdots s_2 s_1 \ \ \ \ \mathrm{and} \ \ \ \ \ \ w_{\Pi} = (s_1 s_2 \cdots s_n)^n.
\]
\begin{proposition}
Let $w = w_{\Pi_n} w_{\Pi}$.  Then $\mathcal{S}(w)^2 = 1$.
\end{proposition}

\begin{proof}
  Let us write $\mathcal{S}(s_i) = r_i$.  To compute $\mathcal{S}(w)$, we must first write $w$ as a reduced expression (see Proposition \ref{braidrelations}).  One may compute that
\[
w = w_{\Pi_l} w_{\Pi} = s_1 s_2 s_1 s_3 s_2 s_1 s_4 s_3 s_2 s_1 \cdots s_{n-1} s_{n-2} \cdots s_2 s_1 (s_1 s_2 \cdots s_n)^n
\]
\[
= s_n s_{n-1} s_n s_{n-2} s_{n-1} s_n s_{n-3} s_{n-2} s_{n-1} s_n \cdots s_1 s_2 s_3 \cdots s_n.
\]
and this last expression is a reduced expression for $w$.  Applying $\mathcal{S}$ to this reduced expression and then squaring the result, we get
\[
 r_n r_{n-1} r_n r_{n-2} r_{n-1} r_n r_{n-3} r_{n-2} r_{n-1} r_n \cdots r_1 r_2 r_3 \cdots r_n  r_n r_{n-1} r_n r_{n-2} r_{n-1} r_n r_{n-3} r_{n-2} r_{n-1} r_n \cdots r_1 r_2 r_3 \cdots r_n
\]
Because of the braid relations, we may rewrite the above product as
\[
r_n r_{n-1} r_n r_{n-2} r_{n-1} r_n r_{n-3} r_{n-2} r_{n-1} r_n \cdots r_1 r_2 r_3 \cdots r_n r_n r_{n-1} \cdots r_1 r_n r_{n-1} \cdots r_2 r_n r_{n-1} \cdots r_3 \cdots r_n r_{n-1} r_n.
\]
We now compute that  
\begin{align*}
& r_1 r_2 r_3 \cdots r_{n-1} r_n r_n r_{n-1} \cdots r_1 = r_1 r_2 r_3 \cdots r_{n-1} \omega_{n-1}(-1) r_{n-1} r_{n-2} \cdots r_1\\
& = r_1 r_2 \cdots r_{n-2} \omega_{n-2}(-1) \omega_{n-1}(-1) r_{n-2} r_{n-3} \cdots r_1\\
& = r_1 r_2 \cdots r_{n-3} \omega_{n-3}(-1) \omega_{n-2}(-1) r_{n-3} r_{n-4} \cdots r_1\\
& = ...\\
& = \omega_1(-1),
\end{align*}
and one may conclude, using the braid relations, that
\[
\mathcal{S}(w)^2 = \omega_1(-1) r_n r_{n-1} r_n r_{n-2} r_{n-1} r_n r_{n-3} r_{n-2} r_{n-1} r_n \cdots r_2 r_3 \cdots r_n r_n r_{n-1} \cdots r_2 r_n r_{n-1} \cdots r_3 \cdots r_n r_{n-1} r_n
\]
Similarly, 
\[
r_2 r_3 \cdots r_n r_n r_{n-1} \cdots r_2 = \omega_1(-1) \omega_2(-1),
\]
giving us
\[
\mathcal{S}(w)^2 = \omega_1(-1) \omega_1(-1) \omega_2(-1) r_n r_{n-1} r_n r_{n-2} r_{n-1} r_n r_{n-3} r_{n-2} r_{n-1} r_n \cdots r_3 \cdots r_n r_n r_{n-1} \cdots r_3 r_n r_{n-1} \cdots r_3 \cdots r_n r_{n-1} r_n
\]
Similarly,
\[
r_3 r_4 \cdots r_n r_n r_{n-1} \cdots r_3 = \omega_2(-1) \omega_3(-1)
\]
\[
r_4 r_5 \cdots r_n r_n r_{n-1} \cdots r_4 = \omega_3(-1) \omega_4(-1)
\]
and so on.  In the end, we obtain $\mathcal{S}(w)^2 = 1$, using the fact that $r_n^2 = \omega_{n-1}(-1)$.
\end{proof}

We now state the main result.

\begin{theorem}\label{kottwitz}
Let $G$ be a split, almost-simple, connected reductive $p$-adic group, and $\mathcal{S}$ an optimal section.  The map 
$
(\epsilon_i, w_i) \mapsto \epsilon_i(\varpi^{-1}) \mathcal{S}(w_i)
$
 is a homomorphic section of the Kottwitz homomorphism $\kappa_G : G(F) \rightarrow \Omega$.
\end{theorem}

\section{Summary of results}\label{summary}

In this section, we present a table which exhibits all order profiles of sections of the Weyl group.  We assume in this section that $n \geq 6$.  For each type, we list the possibilities for the order profiles in terms of what the orders of the lifts of simple reflections are.  We remind the reader that to each section, we associate a labeled Dynkin diagram.

\[
 \begin{tabu}{||c  c||} 
 \hline
\mathrm{Type}   & \mathrm{Order \ profiles \ of \ sections}    \\ [0.5ex] 
 \hline\hline
\A_n \ \mathrm{simply \ connected}, \ n \ \mathrm{odd}   & \begin{tikzpicture}
\dynkin{A}{ooo.o};
\dynkinLabelRoot{1}{4j}
\dynkinLabelRoot{2}{4j}
\dynkinLabelRoot{3}{4j}
\dynkinLabelRoot{4}{4j}
\end{tikzpicture}    \\ 
 &   j \geq 1      \\ 
 \hline
\A_n \ \mathrm{simply \ connected}, \ n \ \mathrm{even}   & \begin{tikzpicture}
\dynkin{A}{ooo.o};
\dynkinLabelRoot{1}{2j}
\dynkinLabelRoot{2}{2j}
\dynkinLabelRoot{3}{2j}
\dynkinLabelRoot{4}{2j}
\end{tikzpicture}   \\ 
 &   j \geq 1       \\ 
 \hline
\A_n, \ \mathrm{with} \ 1 < a := \# Z(G) < n+1, a \ \mathrm{even} & \begin{tikzpicture}
\dynkin{A}{ooo.o};
\dynkinLabelRoot{1}{4j}
\dynkinLabelRoot{2}{4j}
\dynkinLabelRoot{3}{4j}
\dynkinLabelRoot{4}{4j}
\end{tikzpicture}  \\ 
 &   j \geq 1   
   \\ 
    \hline
\A_n, \ \mathrm{with} \ 1 < a := \# Z(G) < n+1, a \ \mathrm{odd}  & \begin{tikzpicture}
\dynkin{A}{ooo.o};
\dynkinLabelRoot{1}{2j}
\dynkinLabelRoot{2}{2j}
\dynkinLabelRoot{3}{2j}
\dynkinLabelRoot{4}{2j}
\end{tikzpicture}  \\ 
 &   j \geq 1   
   \\ 
 \hline
\A_n \ \mathrm{adjoint}  & \begin{tikzpicture}
\dynkin{A}{ooo.o};
\dynkinLabelRoot{1}{2j}
\dynkinLabelRoot{2}{2j}
\dynkinLabelRoot{3}{2j}
\dynkinLabelRoot{4}{2j}
\end{tikzpicture}    \\ 
 &   j \geq 1     \\ 
 \hline
\B_n \ \mathrm{simply \ connected}, \ n \ \mathrm{even}  & \begin{tikzpicture}
\dynkin{B}{ooo.oo}
\dynkinLabelRoot{1}{4}
\dynkinLabelRoot{2}{4}
\dynkinLabelRoot{3}{4}
\dynkinLabelRoot{4}{4}
\dynkinLabelRoot{5}{2}
\end{tikzpicture}   \\
 &   \begin{tikzpicture}
\dynkin{B}{ooo.oo}
\dynkinLabelRoot{1}{4}
\dynkinLabelRoot{2}{4}
\dynkinLabelRoot{3}{4}
\dynkinLabelRoot{4}{4}
\dynkinLabelRoot{5}{4}
\end{tikzpicture}       \\ 
 \hline
 \B_n \ \mathrm{simply \ connected}, \ n \ \mathrm{odd}  & \begin{tikzpicture}
\dynkin{B}{ooo.oo}
\dynkinLabelRoot{1}{4}
\dynkinLabelRoot{2}{4}
\dynkinLabelRoot{3}{4}
\dynkinLabelRoot{4}{4}
\dynkinLabelRoot{5}{4}
\end{tikzpicture}   \\
 \hline
\B_n \ \mathrm{adjoint}  & \begin{tikzpicture}
\dynkin{B}{ooo.oo}
\dynkinLabelRoot{1}{2}
\dynkinLabelRoot{2}{2}
\dynkinLabelRoot{3}{2}
\dynkinLabelRoot{4}{2}
\dynkinLabelRoot{5}{2}
\end{tikzpicture}  \\
&  \begin{tikzpicture}
\dynkin{B}{ooo.oo}
\dynkinLabelRoot{1}{4}
\dynkinLabelRoot{2}{4}
\dynkinLabelRoot{3}{4}
\dynkinLabelRoot{4}{4}
\dynkinLabelRoot{5}{2}
\end{tikzpicture}   \\
 \hline
 \C_n \ \mathrm{simply \ connected}  & \begin{tikzpicture}
\dynkin{C}{ooo.oo}
\dynkinLabelRoot{1}{2}
\dynkinLabelRoot{2}{2}
\dynkinLabelRoot{3}{2}
\dynkinLabelRoot{4}{2}
\dynkinLabelRoot{5}{4}
\end{tikzpicture}   \\
 &   \begin{tikzpicture}
\dynkin{C}{ooo.oo}
\dynkinLabelRoot{1}{4}
\dynkinLabelRoot{2}{4}
\dynkinLabelRoot{3}{4}
\dynkinLabelRoot{4}{4}
\dynkinLabelRoot{5}{4}
\end{tikzpicture}       \\ 
 \hline
 \C_n \ \mathrm{adjoint}  & \begin{tikzpicture}
\dynkin{C}{ooo.oo}
\dynkinLabelRoot{1}{2}
\dynkinLabelRoot{2}{2}
\dynkinLabelRoot{3}{2}
\dynkinLabelRoot{4}{2}
\dynkinLabelRoot{5}{4}
\end{tikzpicture}   \\
 &   \begin{tikzpicture}
\dynkin{C}{ooo.oo}
\dynkinLabelRoot{1}{4}
\dynkinLabelRoot{2}{4}
\dynkinLabelRoot{3}{4}
\dynkinLabelRoot{4}{4}
\dynkinLabelRoot{5}{4}
\end{tikzpicture}    \\ 
 \hline

 \D_n \ \mathrm{simply \ connected}  & \begin{tikzpicture}
\dynkin{D}{oo.ooo}
\dynkinLabelRoot{1}{4}
\dynkinLabelRoot{2}{4}
\dynkinLabelRoot{3}{4}
\dynkinLabelRoot{4}{4}
\dynkinLabelRoot{5}{4}
\end{tikzpicture}   \\
 \hline
 
 \D_n, \ \mathrm{fundamental \ coweight} \ \omega_1  & \begin{tikzpicture}
\dynkin{D}{oo.ooo}
\dynkinLabelRoot{1}{2}
\dynkinLabelRoot{2}{2}
\dynkinLabelRoot{3}{2}
\dynkinLabelRoot{4}{2}
\dynkinLabelRoot{5}{2}
\end{tikzpicture}   \\
&  \begin{tikzpicture}
\dynkin{D}{oo.ooo}
\dynkinLabelRoot{1}{4}
\dynkinLabelRoot{2}{4}
\dynkinLabelRoot{3}{4}
\dynkinLabelRoot{4}{4}
\dynkinLabelRoot{5}{4}
\end{tikzpicture}    \\
 \hline
 \D_n, n \ \mathrm{even}, \ \mathrm{fundamental \ coweight} \ \omega_{n-1}  & \begin{tikzpicture}
\dynkin{D}{oo.ooo}
\dynkinLabelRoot{1}{4}
\dynkinLabelRoot{2}{4}
\dynkinLabelRoot{3}{4}
\dynkinLabelRoot{4}{4}
\dynkinLabelRoot{5}{4}
\end{tikzpicture}     \\
 \hline
 \D_n, n \ \mathrm{even}, \ \mathrm{fundamental \ coweight} \ \omega_n   & \begin{tikzpicture}
\dynkin{D}{oo.ooo}
\dynkinLabelRoot{1}{4}
\dynkinLabelRoot{2}{4}
\dynkinLabelRoot{3}{4}
\dynkinLabelRoot{4}{4}
\dynkinLabelRoot{5}{4}
\end{tikzpicture}     \\
 \hline
 \D_n \ \mathrm{adjoint} & \begin{tikzpicture}
\dynkin{D}{oo.ooo}
\dynkinLabelRoot{1}{2}
\dynkinLabelRoot{2}{2}
\dynkinLabelRoot{3}{2}
\dynkinLabelRoot{4}{2}
\dynkinLabelRoot{5}{2}
\end{tikzpicture}   \\
&  \begin{tikzpicture}
\dynkin{D}{oo.ooo}
\dynkinLabelRoot{1}{4}
\dynkinLabelRoot{2}{4}
\dynkinLabelRoot{3}{4}
\dynkinLabelRoot{4}{4}
\dynkinLabelRoot{5}{4}
\end{tikzpicture}      \\
 \hline
 \end{tabu}
 \]
 
 \[
 \begin{tabu}{||c   c||} 
 \hline
\mathrm{Type}  & \mathrm{Order \ profiles \ of \ sections}    \\ [0.5ex] 
 \hline\hline
 \F_4  & \begin{tikzpicture}
\dynkin{F}{oooo}
\dynkinLabelRoot{1}{4}
\dynkinLabelRoot{2}{4}
\dynkinLabelRoot{3}{2}
\dynkinLabelRoot{4}{2}
\end{tikzpicture}   \\
&  \begin{tikzpicture}
\dynkin{F}{oooo}
\dynkinLabelRoot{1}{4}
\dynkinLabelRoot{2}{4}
\dynkinLabelRoot{3}{4}
\dynkinLabelRoot{4}{4}
\end{tikzpicture}   \\
 \hline
  \G_2   & \begin{tikzpicture}
\dynkin{G}{oo}
\dynkinLabelRoot{1}{2}
\dynkinLabelRoot{2}{2}
\end{tikzpicture}   \\
&  \begin{tikzpicture}
\dynkin{G}{oo}
\dynkinLabelRoot{1}{2}
\dynkinLabelRoot{2}{4}
\end{tikzpicture}   \\
&  \begin{tikzpicture}
\dynkin{G}{oo}
\dynkinLabelRoot{1}{4}
\dynkinLabelRoot{2}{2}
\end{tikzpicture}   \\
&  \begin{tikzpicture}
\dynkin{G}{oo}
\dynkinLabelRoot{1}{4}
\dynkinLabelRoot{2}{4}
\end{tikzpicture}   \\
 \hline
  \E_8 &  \begin{tikzpicture}
\dynkin{E}{oooooooo}
\dynkinLabelRoot{1}{4}
\dynkinLabelRoot{2}{4}
\dynkinLabelRoot{3}{4}
\dynkinLabelRoot{4}{4}
\dynkinLabelRoot{5}{4}
\dynkinLabelRoot{6}{4}
\dynkinLabelRoot{7}{4}
\dynkinLabelRoot{8}{4}
\end{tikzpicture}  \\
 \hline
  \E_7 \ \mathrm{simply \ connected}  & \begin{tikzpicture}
\dynkin{E}{ooooooo}
\dynkinLabelRoot{1}{4}
\dynkinLabelRoot{2}{4}
\dynkinLabelRoot{3}{4}
\dynkinLabelRoot{4}{4}
\dynkinLabelRoot{5}{4}
\dynkinLabelRoot{6}{4}
\dynkinLabelRoot{7}{4}
\end{tikzpicture}  \\
 \hline
  \E_7 \ \mathrm{adjoint}  & \begin{tikzpicture}
\dynkin{E}{ooooooo}
\dynkinLabelRoot{1}{4}
\dynkinLabelRoot{2}{4}
\dynkinLabelRoot{3}{4}
\dynkinLabelRoot{4}{4}
\dynkinLabelRoot{5}{4}
\dynkinLabelRoot{6}{4}
\dynkinLabelRoot{7}{4}
\end{tikzpicture}  \\
 \hline
 \E_6 \ \mathrm{simply \ connected}  & \begin{tikzpicture}
\dynkin{E}{oooooo}
\dynkinLabelRoot{1}{4}
\dynkinLabelRoot{2}{4}
\dynkinLabelRoot{3}{4}
\dynkinLabelRoot{4}{4}
\dynkinLabelRoot{5}{4}
\dynkinLabelRoot{6}{4}
\end{tikzpicture}  \\
 &  \begin{tikzpicture}
\dynkin{E}{oooooo}
\dynkinLabelRoot{1}{12}
\dynkinLabelRoot{2}{12}
\dynkinLabelRoot{3}{12}
\dynkinLabelRoot{4}{12}
\dynkinLabelRoot{5}{12}
\dynkinLabelRoot{6}{12}
\end{tikzpicture}
  \\
 \hline
 \E_6 \ \mathrm{adjoint}  & \begin{tikzpicture}
\dynkin{E}{oooooo}
\dynkinLabelRoot{1}{4}
\dynkinLabelRoot{2}{4}
\dynkinLabelRoot{3}{4}
\dynkinLabelRoot{4}{4}
\dynkinLabelRoot{5}{4}
\dynkinLabelRoot{6}{4}
\end{tikzpicture}  \\
 \hline

\end{tabu}
\]

\section{Low rank groups}\label{lowrank}

In the previous section, we restricted ourselves to $n \geq 6$ for types $\textbf{A}$ through $\textbf{D}$. In this section, we fill in the details for the low rank groups.  

The classification of order profiles for types $\A$ through $\D$, in \S\ref{summary}, hold for many instances of $n \leq 5$, but not all.  Below we provide the details on the exceptions.  To be clear,  since $\B_4$ adjoint does not appear in the table below, it is then the case that the classification of its sections is the same as in $\B_n$ with $n \geq 6$; there are two sections, the orders on the long roots can all be $2$ or all be $4$, and the order on the short root must be $2$.   We also note that for all of the low rank groups for types $\textbf{A}$ through $\textbf{D}$, we will not explicitly write down the torus elements $t_i$ as we did in \S\ref{sections}. One can explicitly compute these by hand rather easily.

\[
 \begin{tabu}{||c  c||} 
 \hline
\mathrm{Type}   & \mathrm{Order \ profiles \ of \ sections}  \\ [0.5ex] 
 \hline\hline
\A_1 \ \mathrm{simply \ connected}    & \begin{tikzpicture}
\dynkin{A}{o};
\dynkinLabelRoot{1}{4}
\end{tikzpicture}  \\ 
 \hline
\A_1 \ \mathrm{adjoint}    & \begin{tikzpicture}
\dynkin{A}{o};
\dynkinLabelRoot{1}{2}
\end{tikzpicture}  \\  \hline
 \D_3 \ \mathrm{simply \ connected}  & \begin{tikzpicture}
\dynkin{D}{ooo}
\dynkinLabelRoot{1}{4j}
\dynkinLabelRoot{2}{4j}
\dynkinLabelRoot{3}{4j}
\end{tikzpicture}  \\
 &   j \geq 1   
   \\
    \hline
 \D_3 \ \mathrm{adjoint}  & \begin{tikzpicture}
\dynkin{D}{ooo}
\dynkinLabelRoot{1}{2j}
\dynkinLabelRoot{2}{2j}
\dynkinLabelRoot{3}{2j}
\end{tikzpicture}    \\
 &   j \geq 1   
   \\
 \hline
  \D_3 \ \mathrm{fundamental \ coweight} \ \omega_1  & \begin{tikzpicture}
\dynkin{D}{ooo}
\dynkinLabelRoot{1}{2j}
\dynkinLabelRoot{2}{2j}
\dynkinLabelRoot{3}{2j}
\end{tikzpicture}    \\
 &   j \geq 1   
   \\
 \hline

 \D_4, \ \mathrm{fundamental \ coweight} \ \omega_4  & \begin{tikzpicture}
\dynkin{D}{oooo}
\dynkinLabelRoot{1}{2}
\dynkinLabelRoot{2}{2}
\dynkinLabelRoot{3}{2}
\dynkinLabelRoot{4}{2}
\end{tikzpicture}   \\
&  \begin{tikzpicture}
\dynkin{D}{oooo}
\dynkinLabelRoot{1}{4}
\dynkinLabelRoot{2}{4}
\dynkinLabelRoot{3}{4}
\dynkinLabelRoot{4}{4}
\end{tikzpicture}    \\
 \hline
 \D_4, \ \mathrm{fundamental \ coweight} \ \omega_3  & \begin{tikzpicture}
\dynkin{D}{oooo}
\dynkinLabelRoot{1}{2}
\dynkinLabelRoot{2}{2}
\dynkinLabelRoot{3}{2}
\dynkinLabelRoot{4}{2}
\end{tikzpicture}   \\
&  \begin{tikzpicture}
\dynkin{D}{oooo}
\dynkinLabelRoot{1}{4}
\dynkinLabelRoot{2}{4}
\dynkinLabelRoot{3}{4}
\dynkinLabelRoot{4}{4}
\end{tikzpicture}    \\
 \hline
  \end{tabu}
 \]

We now include a list of all of the low rank cases whose $T$-conjugacy classes of sections differs from the general versions of those cases, or whose sections merely have a different form than the general case (as presented in \S\ref{sections}), so that their $T$-conjugacy classes are going to look or be different. This list is as follows: simply connected type $\A_1$, adjoint type $\A_1$, adjoint type $\A_3$, the non simply connected, non adjoint isogeny of type $\A_3$, adjoint type $\C_3$, adjoint type $\C_4$, adjoint type $\B_2$, simply connected $\D_3$, adjoint type $\D_3$, adjoint type $\D_4$, the non simply connected, non adjoint isogeny of type $\D_3$, and finally the isogenies of type $\D_4$ given by the fundamental coweights $\omega_{n-1}$ and $\omega_n$.  It is a fairly straightforward exercise to compute the sections and their $T$-conjugacy classes, in these cases.

\end{document}